\documentclass[12pt]{amsart}

\usepackage{amsfonts}
\usepackage{amssymb}
\usepackage{verbatim}
\usepackage[T1]{fontenc}
\usepackage{graphicx}
\usepackage{amsmath}
\usepackage{xcolor}
\usepackage[letterpaper, left=2.5cm, right=2.5cm, top=2.5cm,
bottom=2.5cm,dvips]{geometry}
\newtheorem{theorem}{Theorem}
\newtheorem{observation}{Observation}
\theoremstyle{plain}

\newtheorem{claim}{Claim}

\newtheorem{conjecture}{Conjecture}
\newtheorem{corollary}{Corollary}

\newtheorem{example}{Example}

\newtheorem{lemma}{Lemma}

\newtheorem{problem}{Problem}
\newtheorem{proposition}{Proposition}
\newtheorem{question}{Question}
\newtheorem{remark}{Remark}

\numberwithin{equation}{section}

\usepackage{tikz}

\usepackage{array,multirow}

\usepackage{xcolor}

\usepackage[backref]{hyperref}
\hypersetup{
	colorlinks,
	linkcolor={blue!60!black},
	citecolor={red!60!black},
	urlcolor={red!60!black}
}

\DeclareMathOperator{\supp}{supp}

\begin{document}
	\title{Shuffle squares and ordered nest-free graphs}\thanks{The manuscript was originally entitled \emph{Graphic images of shuffle squares}}

	\author[J. Grytczuk]{Jaros\l aw Grytczuk}
	\address{Faculty of Mathematics and Information Science, Warsaw University
		of Technology, 00-662 Warsaw, Poland}
	\thanks{The first author was supported by Narodowe Centrum Nauki, grant 2020/37/B/ST1/03298.}
	\email{jaroslaw.grytczuk@pw.edu.pl}
	
	\author[B. Pawlik]{Bart\l omiej Pawlik}
	\address{Institute of Mathematics, Silesian University of Technology, 44-100 Gliwice, Poland}
	\email{bpawlik@polsl.pl}

	\author[A. Ruci\'nski]{Andrzej Ruci\'nski}
	\address{Faculty of Mathematics and Computer Science, Adam Mickiewicz University, 61-614 Pozna\'n, Poland}
	\thanks{The third author was supported by Narodowe Centrum Nauki, grant 2024/53/B/ST1/00164.}
	\email{rucinski@amu.edu.pl}

	\maketitle

\begin{abstract}
	 A \emph{shuffle square} is a word consisting of two \emph{shuffled} copies of the same word. For instance, the Turkish word $\mathtt{\color{red}{ik}\color{blue}{i}\color{red}{li}\color{blue}{kli}}$ (\emph{binary} in English) is a shuffle square, as it can be split into two copies of the word $\mathtt{ikli}$.

We explore a representation of shuffle squares in terms of \emph{ordered nest-free graphs} and demonstrate the usefulness of this approach by applying it to several families of binary words. Among others, we characterize shuffle squares with four and five runs, as well as shuffle squares with all $\mathtt1$-runs of length one (and with the $\mathtt1$'s alternating between the two copies).

In our main result we provide quite general sufficient conditions for a binary word \emph{not} to be a shuffle square.
 In particular, it follows that binary words of the type $(\mathtt{1001})^n$, $n$ odd, are not shuffle squares. We complement it by showing that all other words whose every $\mathtt{1}$-run has length one or two, while every $\mathtt{0}$-run has length two, are shuffle squares.

 We also provide a counterexample to a believable stipulation that binary words of the form $\mathtt1^{m}\mathtt0^{m-2}\mathtt1^{m-4}\cdots$, $m$ odd, are far from being shuffle squares (the distance measured by the minimum number of letters one has to delete in order to turn a word into a shuffle square).

\end{abstract}
	
\section{Introduction}
	
\subsection{Definitions}

Let $k\geqslant1$ be an integer and let $A$ be an alphabet with $|A|=k$. By a~\emph{$k$-ary word} of \emph{length $n$} we mean any sequence $W=w_1\cdots w_n$ with $w_i\in A$ for all $1\leqslant i\leqslant n$. For \emph{binary} words we typically take the~alphabet $\{\mathtt0,\mathtt1\}$ or $\{\mathtt{A},\mathtt{B}\}$, while for bigger alphabets we rather use the set $[k]:=\{\mathtt1,\mathtt2,\dots,\mathtt{k}\}$. The length of a~word $W$ will be denoted by $|W|$.

Any subsequence $W'=w_{i_1}\cdots w_{i_t}$ of a word $W=w_1\cdots w_n$ is called a~\emph{subword} of $W$. Its set of indices is denoted by $\supp(W')=\{i_1,\dots,i_t\}$ and called the~\emph{support} of $W'$. If the~support of a~subword $W'$ consists of consecutive integers, then we call $W'$ a \emph{block} of $W$.

A pair of subwords $X,Y$ of $W$ with $X=Y$ and disjoint supports is called \emph{twins}. We will be using the shorthand $$\supp(X,Y)=\supp(X)\cup \supp(Y)$$ wherever convenient.
The~length of twins is defined as the length of either one of them, $|X|=|Y|$. We will occasionally  refer also to the \emph{double length} of twins which equals twice the length and measures how much of the word is taken by the twins. The letters of $W$ which do not belong to either of the twins are called \emph{gaps} (with respect to $X,Y$).

\begin{example}\label{small}\rm Let $$W=\mathtt{111001000110}$$
be a binary word of length $12$. One pair of twins $X,Y$ of length 5 (so of double length $10$) in $W$ is given by $\supp(X)=\{1,6,8,9,12\}$ and $\supp(Y)=\{2,3,4,5,7\}$, as indicated by colors and under/overlines: $$W=\underline{\textcolor{red}{\mathtt{1}}}\;\overline{\textcolor{blue}{\mathtt{1}}}\overline{\textcolor{blue}{\mathtt{1}}}\overline{\textcolor{blue}{\mathtt{0}}}
\overline{\textcolor{blue}{\mathtt{0}}} \; \underline{\textcolor{red}{\mathtt{1}}} \;
\overline{\textcolor{blue}{\mathtt{0}}} \; \underline{\textcolor{red}{\mathtt{0}}}\underline{\textcolor{red}{\mathtt{0}}} \;
\mathtt{1}\mathtt{1} \; \underline{\textcolor{red}{\mathtt{0}}}.$$
Indeed, we have $X=Y=\mathtt{11000}$, while $w_{10}$ and $w_{11}$ are gaps here.
\end{example}

Given two words $U={u}_1\cdots {u}_n$ and $W={w}_1\cdots {w}_m$, their \emph{concatenation} is defined as the~word $UW={u}_1\cdots {u}_n {w}_1\cdots {w}_m$. This definition extends naturally to any finite number of words. In particular, $W^r$ denotes the concatenation of $r$ copies of the word $W$. For $w\in A$, we may write ${w}^r$ instead of ${ww}\cdots{w}$. For instance,
 $$W=\mathtt{111001000110} = \mathtt1^3\mathtt0^2\mathtt{10}^3\mathtt1^2\mathtt0.$$ This notation comes in handy when there are long blocks of identical letters.

  For any letter $w\in A$, a {\it $w$-run} in a word $W$ is a maximal block $F$ of $W$ of the form $F={w}^r$ for some $r\geqslant1$. Here ``maximal'' means that no other block $F'$ of $W$ which contains $F$ is of the form $F'=w^{r'}$, with $r'>r$. By a \emph{run} we mean a~$w$-run for any $w\in A$. Every word can be uniquely decomposed into a concatenation of its runs. For instance, in $\mathtt{011011100}$ there are two $\mathtt1$-runs ($\mathtt{11},\mathtt{111}$) and three $\mathtt0$-runs ($\mathtt{0},\mathtt{0},\mathtt{00}$), and the corresponding decomposition is $$\mathtt{0\;11\;0\;111\;00}=\mathtt{0}\mathtt{1}^2\mathtt{0}\mathtt{1}^3\mathtt{0}^2.$$

For every non-empty word $X$ we call the word $X^2=XX$ a {\it square}. For instance, the English word $\mathtt{\color{red}{hots}\color{blue}{hots}}$ is a square. This simple structure is a frequent theme of investigations since the celebrated 1906 result of Thue \cite{Thue} (see the next subsection).
	In this paper, however, we focus on a natural generalization of squares created more frivolously from two copies of the same word.
A \emph{shuffle square} is a word $W$ which contains twins of double length equal to exactly $|W|$; such twins will be called \emph{perfect}, as they fill the whole word with no gaps. For instance, $\mathtt{00001001}$ is a shuffle square with the~perfect twins equal to $\mathtt{0001}$: $$\mathtt{\underline{\textcolor{red}{0}}\overline{\textcolor{blue}{0001}}\underline{\textcolor{red}{001}}}.$$ Obviously, a necessary condition for a word to be a~shuffle square is that each letter in the~word occurs an even number of times. We will call such words \emph{even} (elsewhere, they are called \emph{tangrams}, cf. \cite{Deb2024,Och2025}).

\subsection{Related results and conjectures}

Shuffle squares were introduced in 2012 by Henshall, Rampersad, and Shallit \cite{Hen2012}, and since then their various aspects have been intensively studied (see \cite{Bul2020,Bul2023,Buss-Soltys,Gry2024,He2021}).

The main challenge raised in \cite{Hen2012} is to determine the number $S_k(n)$ of distinct $k$-ary shuffle squares of length $2n$. The problem seems hard even in the most basic case of $k=2$.  He, Huang, Nam, and Thaper \cite{He2021} derived an asymptotic formula for $S_k(n)$, stipulated in \cite{Hen2012}, and  proved that $S_2(n)\geqslant \binom{2n}{n}$. They also posed an intriguing conjecture stating that almost all \emph{even} binary words are shuffle squares.   Since exactly a half of binary words of fixed even length are even, the conjecture is equivalent to
\begin{equation}\label{Limit HHNT-Conjecture}
	\lim_{n\rightarrow \infty}\frac{S_2(n)}{2^{2n-1}}=1.
\end{equation}

For a given binary word $W$, let us denote by $g(W)$ the minimum number of gaps in $W$, that is, letters one has to delete from $W$ in order to get a shuffle square. For instance, $g(\mathtt{0011})=0$, while $g(\mathtt{0110})=2$. Let $g(n)$ be the maximum of $g(W)$ over all binary words $W$ of fixed length $n$. 

By a sophisticated approach via a regularity lemma for words, Axenovich, Person, and Puzynina \cite{Axenovich-Person-Puzynina} showed that $g(n)=o(n)$. This means that every binary word of length $n$ contains a shuffle square of length asymptotically close to $n$. On the other hand, by an~explicit construction they showed that $g(n)=\Omega(\log n)$, which was improved by Bukh \cite{Bukh} to $g(n)=~\Omega(n^{1/3})$ (see \cite{BR} for a~proof). The exact asymptotic order of $g(n)$ remains a mystery, though some intuition suggests that it should be $g(n)=\Theta(n^{1/2})$.

In the spirit of Thue, one may also investigate various \emph{avoidability} properties involving shuffle squares. Generally, a set of words $\mathcal F$ is \emph{$k$-avoidable} if there exist arbitrarily long $k$-ary words with no blocks belonging to $\mathcal F$. Thue proved in \cite{Thue} that squares are $3$-avoidable. This result is considered the starting point of Combinatorics on Words, a wide discipline with plenty of connections and applications in various areas (see \cite{BerstelPerrin}, \cite{Lothaire}).

What is the smallest $k$ such that shuffle squares are $k$-avoidable? Using the probabilistic method, Currie \cite{Currie} proved that $k\leqslant 10^{40}$. Later, this was improved to $k\leqslant 10$ by M\"{u}ller~\cite{Muller}, and to $k\leqslant 7$, by Gu\'egan and Ochem \cite{GueganOchem} and, independently, by Grytczuk, Kozik, and Zaleski \cite{GrytczukKozikZaleski}. The last result is true in a more complex \emph{list} variant, where the letters of constructed words have to be chosen from  alphabets assigned individually to each position. Currently the best result is due to Bulteau, Jug\'e, and Vialette \cite{Bul2023}, who proved that shuffle squares are $6$-avoidable. They also conjecture that $k=4$, which would be best possible.

On the algorithmic side, Bulteau and Vialette \cite{Bul2020} proved that deciding if a given binary word is a shuffle square is NP-complete. The proof reduces the problem to \emph{Bin-Packing} via graph structures similar to the ones we introduce and explore in this paper.

In the last section we present more open problems related to shuffle squares.

\subsection{Our results and the method}
In this paper we first characterize shuffle squares in terms of ordered nest-free graphs (see Proposition~\ref{Proposition Characterization}). Then, based on this characterization,
we establish conditions under which binary words belonging to specified classes are shuffle squares and conditions under which they are not.

In Section~\ref{2pieces}, as a gentle warm-up,  we consider  words with four runs and five runs, as well as  words with all $\mathtt1$-runs of length one (see, respectively, Propositions~\ref{abcd},~\ref{abcde} and~\ref{only1}).

In Section~\ref{sec-abba}, we prove our main result.

 \begin{theorem}\label{abba}
For an integer $n\geqslant 1$, let $(a_1,\ldots,a_{n+1})$ and $(b_1,\ldots,b_n)$ be two sequences of positive integers such that:
\begin{itemize}
	\item [(1)] $a_1$ and $a_{n+1}$ are odd, while $a_2,\ldots,a_{n}$ are even,
	\item [(2)] for every $I\subset[n]$, $\sum_{i\in I}b_i\neq \sum_{i\not\in I}b_i$, that is, the sequence $(b_1,\ldots,b_n)$ cannot be split into two disjoint subsequences with the same sum.
\end{itemize}
Then the word  $$W=\mathtt{1}^{a_1}\mathtt{0}^{b_1}\mathtt{1}^{a_2}\mathtt{0}^{b_2}\mathtt{1}^{a_3}\cdots\mathtt{1}^{a_n}\mathtt{0}^{b_n}\mathtt{1}^{a_{n+1}}$$ is \emph{not} a shuffle square. 
\end{theorem}

 It follows (see Corollary~\ref{abba-cor} with $r=2$ in Section~\ref{sec-abba}) that the words $(\mathtt{1001})^n$, for $n$ odd, are \emph{not} shuffle squares. This was conjectured by Komisarski \cite{Kom}, who used to call it the \emph{odd ABBA problem}. Let us also mention that an alleged proof was presented by Fortuna \cite{For}, though, up to our knowledge, it has not yet been written down. 
Actually, we show (see Proposition~\ref{Ths1}) that the words $(\mathtt{1001})^n$, $n$ odd, are the \emph{only} binary words with all $\mathtt0$-runs of length two and all $\mathtt1$-runs of length at most two which are not shuffle squares.

\begin{example}\rm One class of instances which trivially satisfy  assumption (2) of Theorem~\ref{abba} is when $2\max_i b_i>\sum_i b_i$. Here is one example of such a word: $$\mathtt1\;\mathtt0^2\;\mathtt1^2\;\mathtt0^2\;\mathtt1^2\;\mathtt0^{10}\;\mathtt1^2\;\mathtt0^2\;\mathtt1^2\;\mathtt0^2\;\mathtt1.$$
\end{example}

\begin{example}\rm Let us notice that conditions (1) and (2) of Theorem \ref{abba} do not characterize binary words which are not shuffle squares. Consider the word $W=\mathtt{101100110001}$. The~related sequences of capacities are $(1,2,2,1)$ and $(1,2,3)$, so  condition (1) is satisfied while (2)~is~not, since $1+2=3$. Nevertheless, $W$ is not a shuffle square, as it can be verified by inspection (or by the next-free graph method introduced in this paper).

\end{example}

In the last section, we tackle the question of how far a binary word may be from a shuffle square. Inspired by a personal communication with Boris Bukh and his notes \cite{Bukh}, in \cite{BR}   the~following conjecture has been implicitly stated.

\begin{conjecture}\label{odd}
There is a constant $c$ such that for every $n$ we have $g(n)\geqslant c\sqrt n$, that is, there exists a binary word
of length $n$ with largest twins leaving out at least $c\sqrt n$
elements.
\end{conjecture}

In view of earlier constructions, we had thought for some time that a natural candidate to facilitate Conjecture~\ref{odd} was the binary word $$O_{m,r}=\mathtt1^m\mathtt0^{m-2}\mathtt1^{m-4}\mathtt0^{m-6}\mathtt1^{m-8}\cdots$$ consisting of $r$ runs whose lengths form a decreasing sequence of consecutive odd numbers, for some odd $m$ and $r\leqslant (m+1)/2$. Here we rebuke this belief by showing that $g(O_{m,r})\le 23$, that is, by finding twins in $O_{m,r}$ of double length at least $|O_{m,r}|-23$, for all $m$ and $r$.

But, in our opinion, at least as important as the obtained results is the method we propose to prove them with. To each pair of \emph{canonical} twins we uniquely assign an \emph{ordered graph} (see the formal definitions in the next sections). Thus, instead of looking for twins in a word or deciding if it is a shuffle square, we construct a corresponding graph or show that such a graph does not exist. We believe that this approach is more transparent, insightful, and ultimately has some potential in this area of research.
	
This paper constitutes a full (journal) version of a conference extended abstract \cite{WORDS}. Compared with \cite{WORDS}, it contains new results (Theorem~\ref{abba}, Proposition~\ref{abcde}, Corollary~\ref{1and2}, and Claim~\ref{cl}), the proof of Proposition~\ref{Ths1} which was omitted in \cite{WORDS}, a new subsection~\ref{cantwins}, as well as more figures and examples.

\subsection{Canonical twins}\label{cantwins}

Before we turn to graphic representation of shuffle squares and, more generally, twins in words, we need to introduce an important notion of canonical twins. But first we define monotone twins. Twins $X,Y$ with supports $\supp(X)=\{i_1,\dots,i_t\}$ and $\supp(Y)=\{j_1,\dots,j_t\}$ are \emph{monotone} if for every $h\in[t]$, we have $i_h<j_h$.  By definition, among monotone twins there is always the ``first'' twin and the ``second'' twin, so we may now treat them as ordered pairs and denote by $(X,Y)$.

Observe that if  in a word $W$ there are twins $X,Y$ of length $t$, then there are also monotone twins $(X',Y')$ of length $t$
(and also with $\supp(X',Y')=\supp(X,Y)$ and $X'=X$).
We obtain them by swapping the elements which appear in $X$ and $Y$ in the wrong order.
 More precisely, let $S=\{h: i_h>j_h\}$.
Then,  the sets
$$I'=\{i_h: h\notin S\}\cup\{j_h: h\in S\}\quad\mbox{and}\quad J'=\{j_h: h\notin S\}\cup\{i_h: h\in S\}$$
 support monotone twins.

For instance, in Example~\ref{small} from the previous section, the twins are not monotone, but with $S=\{2,3,4,5\}$, that is after swapping four pairs of elements between $X$ and $Y$ (all but the first one), the pair of twins $(X',Y')$, indicated here:
$$W=\mathtt{\underline{\textcolor{red}{1}}\;\overline{\textcolor{blue}{1}}\;\underline{\textcolor{red}{1}}\underline{\textcolor{red}{0}}
\underline{\textcolor{red}{0}}\;\overline{\textcolor{blue}{1}}\;
\underline{\textcolor{red}{0}}\;\overline{\textcolor{blue}{0}}\overline{\textcolor{blue}{0}}\;
11\;\overline{\textcolor{blue}{0}}},$$
is monotone.

Besides swapping there are other operations which do not change the twins except for their location and mutual entanglement. Given $i_g,i_h\in \supp(X)$ and $j_g,j_h\in \supp(Y)$, $g<h$, with $j_g<i_h$, $w_{i_h}=w_{j_g}$, and $$\{j_g+1,\dots,i_h-1\}\cap \supp(X,Y)=\emptyset,$$ the \emph{$(g,h)$-rewiring} of the~\emph{monotone} twins $(X,Y)$ in $W$ is a new pair of twins, $(X',Y')$ with $$\supp(X')=(\supp(X)\setminus\{i_h\})\cup\{j_g\}$$ and $$\supp(Y')=(\supp(Y)\setminus\{j_g\})\cup\{i_h\}.$$ Note that $(X',Y')$ are still monotone.
Monotone twins $(X,Y)$ for which a rewiring operation is no longer possible, enjoy the property that within each run all elements of $X$ precede all elements of $Y$.

Finally, a trivial operation of \emph{shifting}  moves all the gaps to the right end of each run. It applies whenever there is a pair of consecutive positions, $\{i,i+1\}$, such that $i$ is a gap while, say, $i+1\in \supp(X)$. Then we switch them around, that is, for the new $X'$ we have $$\supp(X')=(\supp(X)\setminus\{i+1\})\cup\{i\}.$$ When $i+1\in \supp(Y)$, the operation is analogous. If this operation is also no longer possible, then the resulted twins  partition each run of $W$ into three consecutive blocks (some might be empty) with the elements of $X$ to the left, followed by the elements of $Y$, followed by gaps.
Such twins will be called \emph{canonical}.
It is crucial for us to observe that among the longest twins in a word there is always (at least) one pair which is canonical.

\begin{example}\label{oper}\rm
 The twins in $$W=\mathtt{\underline{\textcolor{red}{1}}\; \overline{\textcolor{blue}{1}}\;\underline{\textcolor{red}{1}}\underline{\textcolor{red}{0}}
\underline{\textcolor{red}{0}}\;\overline{\textcolor{blue}{1}}\;
\underline{\textcolor{red}{0}}\;\overline{\textcolor{blue}{0}}\overline{\textcolor{blue}{0}}\;
11\;\overline{\textcolor{blue}{0}}}$$
are not canonical, due to the (dis)order in the first 1-run. A single rewiring operation, exactly the $(1,2)$-rewiring, makes it canonical:
$$W=\mathtt{\underline{\textcolor{red}{1}}\underline{\textcolor{red}{1}}\;\overline{\textcolor{blue}{1}}\;\underline{\textcolor{red}{0}}
\underline{\textcolor{red}{0}}\;\overline{\textcolor{blue}{1}}\;
\underline{\textcolor{red}{0}}\;\overline{\textcolor{blue}{0}}\overline{\textcolor{blue}{0}}\;
11\;\overline{\textcolor{blue}{0}}}$$ (see Fig. \ref{fig1}).

\begin{figure}[h]
\centering
\begin{tikzpicture}
\coordinate (P1) at (0, 0);
\coordinate (P2) at (1, 0);
\coordinate (P3) at (2, 0);
\coordinate (P4) at (3, 0);
\coordinate (P5) at (4, 0);
\coordinate (P6) at (5, 0);
\coordinate (P7) at (6, 0);
\coordinate (P8) at (7, 0);
\coordinate (P9) at (8, 0);
\coordinate (P10) at (9, 0);
\coordinate (P11) at (10, 0);
\coordinate (P12) at (11, 0);

\foreach \P in {P1, P2, P3, P4, P5, P6, P7, P8, P9, P10, P11, P12}
 {\fill (\P) circle (2pt);}

\draw[-, red,dashed,bend left=90] (P1) to (P3);
\draw[-, red,dashed,bend left=90] (P2) to (P6);
\draw[-,bend left=90] (P1) to (P2);
\draw[-, bend left=90] (P3) to (P6);
\draw[-, bend left=90] (P4) to (P8);
\draw[-, bend left=90] (P5) to (P9);
\draw[-, bend left=90] (P7) to (P12);

\end{tikzpicture}
\caption{Non-canonical and canonical (red dashed)  twins in the word $W=~\mathtt{111001000110}$. The arcs join the $i$-th element of $X$ with the corresponding $i$-th element of $Y$, $i=1,\dots,5$.}\label{fig1}
\end{figure}

In another example, $$W=\mathtt{\underline{\textcolor{red}{0}}\;1\;\underline{\textcolor{red}{1}}\;\overline{\textcolor{blue}{0}}\overline{\textcolor{blue}{1}}
\;\underline{\textcolor{red}{1}}\;\overline{\textcolor{blue}{1}}\;\underline{\textcolor{red}{0}}\;\overline{\textcolor{blue}{0}}},$$ the twins are not canonical for two reasons: a gap in front of the first $\mathtt{1}$-run and an element of $Y$ preceding an element of $X$ in  the second $\mathtt{1}$-run. Two operations, one shifting and one rewiring brings it to the canonical form
$$W=\mathtt{\underline{\textcolor{red}{0}}\underline{\textcolor{red}{1}}\;1\;\overline{\textcolor{blue}{0}}\;\underline{\textcolor{red}{1}}\;\overline{\textcolor{blue}{1}}
\overline{\textcolor{blue}{1}}\;\underline{\textcolor{red}{0}}\;\overline{\textcolor{blue}{0}}}$$ (see Fig. \ref{fig2}).

\begin{figure}[h]
\centering
\begin{tikzpicture}
\coordinate (P1) at (0, 0);
\coordinate (P2) at (1, 0);
\coordinate (P3) at (2, 0);
\coordinate (P4) at (3, 0);
\coordinate (P5) at (4, 0);
\coordinate (P6) at (5, 0);
\coordinate (P7) at (6, 0);
\coordinate (P8) at (7, 0);
\coordinate (P9) at (8, 0);

\foreach \P in {P1, P2, P3, P4, P5, P6, P7, P8, P9}
 {\fill (\P) circle (2pt);}

\draw[-, red,dashed,bend left=90] (P2) to (P6);
\draw[-, red,dashed,bend left=90] (P5) to (P7);
\draw[-,bend left=90] (P1) to (P4);
\draw[-, bend left=90] (P3) to (P5);
\draw[-, bend left=90] (P6) to (P7);
\draw[-, bend left=90] (P8) to (P9);;

\node[below=5pt] at (P1) {$\mathtt{0}$};
\node[below=5pt] at (P2) {$\mathtt{1}$};
\node[below=5pt] at (P3) {$\mathtt{1}$};
\node[below=5pt] at (P4) {$\mathtt{0}$};
\node[below=5pt] at (P5) {$\mathtt{1}$};
\node[below=5pt] at (P6) {$\mathtt{1}$};
\node[below=5pt] at (P7) {$\mathtt{1}$};
\node[below=5pt] at (P8) {$\mathtt{0}$};
\node[below=5pt] at (P9) {$\mathtt{0}$};

\end{tikzpicture}
\caption{Non-canonical and canonical (red dashed) twins in the word $W=\mathtt{011011100}$.}\label{fig2}
\end{figure}

\end{example}

\section{Graphic representations}\label{gr}

Here we introduce our main tool---a representation of shuffle squares and, more generally, canonical twins in terms of ordered graphs with no nests.
A similar idea was utilized in \cite{Bul2020} as a reduction tool in the proof of the NP-completeness of the decision problem for shuffle squares.

\subsection{Ordered graphs from words}

An \emph{ordered graph} is just a graph with its vertex set linearly ordered, say \emph{from left to right}. We allow loops and parallel edges. A \emph{nest} in an ordered graph consists of a pair of disjoint edges (including loops) $e$ and $f$ with $\min e<\min f$ and $\max e>\max f$ (see Fig. \ref{fig3}). An ordered graph is \emph{nest-free} if no pair of its edges forms a nest.

\begin{figure}[h]
\centering
\begin{tikzpicture}
\coordinate (P1) at (0, 0);
\coordinate (P2) at (1, 0);
\coordinate (P3) at (2, 0);
\coordinate (P4) at (3, 0);

\draw[-,bend left=90] (P1) to node[midway, above]{$e$} (P4);
\draw[-,bend left=90] (P2) to node[midway, below]{$f$} (P3);

\foreach \P in {P1, P2, P3, P4}
 {\fill (\P) circle (2pt);}
\end{tikzpicture}
\caption{A nest in an ordered graph.}\label{fig3}
\end{figure}

Given an ordered graph $G$ and a vertex $u\in V(G)$, the \emph{right degree} $\deg_G^{(\rightarrow)}(u)$ of $u$ in $G$ is defined as the number of edges going from $u$ to the right, and the \emph{left degree} $\deg_G^{(\leftarrow)}(u)$ is defined analogously (each loop at $u$ contributes 1 to  $\deg_G^{(\rightarrow)}(u)$ and 1 to $\deg_G^{(\leftarrow)}(u)$). Obviously, we have $$\deg_G(u)=\deg_G^{(\rightarrow)}(u)+\deg_G^{(\leftarrow)}(u)$$  for every $u\in V(G)$

We now describe a (unique) representation of canonical twins as a nest-free graph. Let a $k$-ary word $W=w_1\cdots w_n$ have canonical twins $(X,Y)$ where $X=w_{i_1}\cdots w_{i_t}$ and $Y=~w_{j_1}\cdots w_{j_t}$. In particular, $$\{i_1,\dots, i_t\}\cap\{j_1,\dots, j_t\}=\emptyset,$$ and for all $h=1,\dots,t$, we have $i_h<j_h$ and $w_{i_h}=w_{j_h}$. As an intermediate step, define an ordered graph $M=M_{X,Y}$ on the set of vertices $\{1,\ldots,n\}$ with edges $\{i_1,j_1\},\dots,\{i_t,j_t\}$. So, $M$ consists of a \emph{matching} of size $t$ and $n-2t$ isolated vertices. 
 Clearly, there are no nests in $M$.

Let $W=U_1\cdots U_m$, where each block $U_h$ is a run in $W$.
Contract each set $\supp(U_h)$ to a~single vertex $u_h$, keeping all the edges of $M$ in place and keeping the order $u_1<\cdots<u_m$. We denote the obtained ordered graph by $G=G_{X,Y}$.
Observe that $G$ is still nest-free and, moreover, consists of $k$ separated graphs on vertex sets $N_w=\{h:\text{$U_h$ is a $w$-run}\}$, one for each letter $w\in [k]$.
The degrees of vertices in $G$ are $$\deg_G(u_h)=|\supp(U_h)\cap \supp(X,Y)|$$ for $h=1,\dots,m$ (see Examples~\ref{ter} and~\ref{small1}).

 From $G$ it is straightforward to uniquely reconstruct $(X,Y)$. The uniqueness is a consequence of the fact that we restrict ourselves to canonical twins. We simply scan the vertices $u_h$ of $G$ from left to right and record their right degrees $\deg_G^{(\rightarrow)}(u_h)$  and left degrees $\deg_G^{(\leftarrow)}(u_h)$. Then from each run $U_h$, assign the first $\deg_G^{(\rightarrow)}(u_h)$ elements to $X$ and the next $\deg_G^{(\leftarrow)}(u_h)$ to $Y$, leaving the remaining vertices of $U_h$, if there are any, unassigned (gaps).
But if we are only interested in the length of the twins, we may infer that there is a pair of twins in $W$ of double length equal to $\sum \deg_G(u_h)$.

This allows us to apply this ``graphic'' approach also in situations when our goal is to find  as large as possible twins in $W$. Then we simply attempt to construct a nest-free ordered graph $G$ on vertices $u_h$, $h=1,\dots,m$, with no edges joining vertices from different sets $N_w$, and with degrees as close as possible to (but not greater than) the lengths of the corresponding runs $U_h$. We call $|U_h|$ the \emph{capacity} of vertex $u_h$. In particular, if we succeed to create such a~graph $G$ with all $\deg_G(u_h)=|U_h|$, then we know that $W$ is a shuffle square and we can produce a pair of perfect twins witnessing this property.

Let us summarize the above discussion in the following statement characterizing shuffle squares in terms of their ordered graph representations.

\begin{proposition}\label{Proposition Characterization}
Let $W$ be a $k$-ary word and let $W=U_1\cdots U_m$ be a~decomposition of $W$ into its runs. Then $W$ is a shuffle square if and only if there exists an ordered graph $G$ on the set $\{u_1,\dots, u_m\}$ satisfying the following conditions.
\begin{itemize}
	\item[1.] If $u_i$ is joined to $u_j$ by an edge, then the corresponding runs $U_i$ and $U_j$ are $w$-runs for the same letter $w$.
	\item[2.] The degree of each vertex $u_i$ satisfies $\deg_G(u_i)=|U_i|$.
	\item[3.]  $G$ is nest-free.
\end{itemize}
Moreover, if there is a graph $G$ satisfying properties 1-3, then one can construct (canonical) perfect twins in $W$ by assigning the first $\deg_G^{(\rightarrow)}(u_h)$ elements of every run $U_h$ to $X$ and the~next $\deg_G^{(\leftarrow)}(u_h)$ elements to $Y$.
\end{proposition}

\subsection{Examples of nest-free graphs on twins}

Let us now illustrate the nest-free graph approach to twins by a number of examples. In all figures, for clarity of exposition,  we adopt the convention that, given a word $W$, for each $w\in A$ the vertices in the set $N_w$ appear in one row, and the rows corresponding to different letters are drawn one beneath another (in the same order the letters appear in the word), but so that the global linear order is preserved. Also, in some figures next to each vertex we place (in parenthesis) its capacity.
\begin{example}\label{ter}\rm
The ternary word $$W=\underline{\textcolor{red}{\mathtt1}}\underline{\textcolor{red}{\mathtt1}}\;\overline{\textcolor{blue}{\mathtt1}}\;\underline{\textcolor{red}{\mathtt2}}
\underline{\textcolor{red}{\mathtt2}}
\underline{\textcolor{red}{\mathtt3}}\;\mathtt3\mathtt3\;
\overline{\textcolor{blue}{\mathtt1}}
\overline{\textcolor{blue}{\mathtt2}}\overline{\textcolor{blue}{\mathtt2}}
\overline{\textcolor{blue}{\mathtt3}}$$
has canonical twins of length 5 as indicated. The~corresponding matching $M$ and graph $G$ are given in Fig. \ref{fig4}.

Is $W$ a shuffle square? If it was, there would exist a graph satisfying  conditions 1-3 in Proposition~\ref{Proposition Characterization}. However,
there is no way to change $G$, so that $\deg_G(u_3)=3$, while keeping all other degrees intact. Indeed, $u_6$ has to be adjacent to $u_3$, so there cannot be a loop at $u_5$ which, in turn, implies that there is no loop at $u_3$ either.

If we swapped the $\mathtt{3}$-runs around and instead consider the word $W'=\mathtt{111223122333}$, then the graph in Fig. \ref{fig5} shows that $W'$ is a shuffle square:
$$W'=\underline{\textcolor{red}{\mathtt1}}\underline{\textcolor{red}{\mathtt1}}\;\overline{\textcolor{blue}{\mathtt1}}\;\underline{\textcolor{red}{\mathtt2}}
\underline{\textcolor{red}{\mathtt2}}
\underline{\textcolor{red}{\mathtt3}}\;
\overline{\textcolor{blue}{\mathtt1}}
\overline{\textcolor{blue}{\mathtt2}}\overline{\textcolor{blue}{\mathtt2}}\;\underline{\textcolor{red}{\mathtt3}}\;
\overline{\textcolor{blue}{\mathtt3}}\overline{\textcolor{blue}{\mathtt3}}.$$
\end{example}

\begin{figure}[h]
\centering
\begin{tikzpicture}[scale=.9]
\coordinate (P1) at (0, 0);
\coordinate (P2) at (1, 0);
\coordinate (P3) at (2, 0);
\coordinate (P4) at (3, 0);
\coordinate (P5) at (4, 0);
\coordinate (P6) at (5, 0);
\coordinate (P7) at (6, 0);
\coordinate (P8) at (7, 0);
\coordinate (P9) at (8, 0);
\coordinate (P10) at (9, 0);
\coordinate (P11) at (10, 0);
\coordinate (P12) at (11, 0);

\node[above left=15pt] at (P1) {$M$:};

\draw[draw=green,thick,dashed,rounded corners] (-.2, -.2) rectangle (2.2, .2);
\draw[draw=green,thick,dashed,rounded corners] (2.8, -.2) rectangle (4.2, .2);
\draw[draw=green,thick,dashed,rounded corners] (4.8, -.2) rectangle (7.2, .2);
\draw[draw=green,thick,dashed,rounded corners] (7.8, -.2) rectangle (8.2, .2);
\draw[draw=green,thick,dashed,rounded corners] (8.8, -.2) rectangle (10.2, .2);
\draw[draw=green,thick,dashed,rounded corners] (10.8, -.2) rectangle (11.2, .2);

\foreach \P in {P1, P2, P3, P4, P5, P6, P7, P8, P9,P10,P11,P12}
 {\fill (\P) circle (2pt);}

\draw[-,bend left=90] (P1) to (P3);
\draw[-, bend left=90] (P2) to (P9);
\draw[-, bend left=40] (P4) to (P10);
\draw[-, bend left=90] (P5) to (P11);
\draw[-, bend left=40] (P6) to (P12);

\node[below=5pt] at (P1) {$1$};
\node[below=5pt] at (P2) {$2$};
\node[below=5pt] at (P3) {$3$};
\node[below=5pt] at (P4) {$4$};
\node[below=5pt] at (P5) {$5$};
\node[below=5pt] at (P6) {$6$};
\node[below=5pt] at (P7) {$7$};
\node[below=5pt] at (P8) {$8$};
\node[below=5pt] at (P9) {$9$};
\node[below=5pt] at (P10) {$10$};
\node[below=5pt] at (P11) {$11$};
\node[below=5pt] at (P12) {$12$};

\coordinate (U1) at (1, -2);
\coordinate (U2) at (8, -2);
\coordinate (U3) at (3.5, -3);
\coordinate (U4) at (9.5, -3);
\coordinate (U5) at (6, -4);
\coordinate (U6) at (11, -4);

\node[left=120pt] at (U3) {$G$:};

\draw[-] (U1) to (U2);
\draw[-] (U1) to[in=50,out=130, distance=1.5cm] (U1);
\draw[-, bend left=5] (U3) to (U4);
\draw[-, bend left=5] (U4) to (U3);
\draw[-] (U5) to (U6);

\foreach \P in {U1, U2, U3, U4, U5, U6}
 {\fill [fill=green, draw=black, line width=0.5pt] (\P) circle (2.5pt);}

 \node[below=5pt] at (U1) {$u_1(3)$};
\node[below=5pt] at (U2) {$u_4(1)$};
\node[below=5pt] at (U3) {$u_2(2)$};
\node[below=5pt] at (U4) {$u_5(2)$};
\node[below=5pt] at (U5) {$u_3(3)$};
\node[below=5pt] at (U6) {$u_6(1)$};
\end{tikzpicture}
\caption{Matching $M$ and graph $G$ for $W=\underline{\textcolor{red}{\mathtt{11}}}\;\overline{\textcolor{blue}{\mathtt1}}\;\underline{\textcolor{red}{\mathtt2}}
\underline{\textcolor{red}{\mathtt2}}
\underline{\textcolor{red}{\mathtt3}}\;\mathtt3\mathtt3\;
\overline{\textcolor{blue}{\mathtt1}}
\overline{\textcolor{blue}{\mathtt2}}\overline{\textcolor{blue}{\mathtt2}}
\overline{\textcolor{blue}{\mathtt3}}$.}\label{fig4}
\end{figure}

\begin{figure}[h]
\centering
\begin{tikzpicture}
\coordinate (U1) at (1, -2.5);
\coordinate (U2) at (6, -2.5);
\coordinate (U3) at (3.5, -3.5);
\coordinate (U4) at (7.5, -3.5);
\coordinate (U5) at (5, -4.5);
\coordinate (U6) at (10, -4.5);

\draw[-] (U1) to (U2);
\draw[-] (U1) to[in=50,out=130, distance=1.5cm] (U1);
\draw[-, bend left=5] (U3) to (U4);
\draw[-, bend left=5] (U4) to (U3);
\draw[-] (U5) to (U6);
\draw[-] (U6) to[in=50,out=130, distance=1.5cm] (U6);

\foreach \P in {U1, U2, U3, U4, U5, U6}
 {\fill [fill=green, draw=black, line width=0.5pt] (\P) circle (2.5pt);}

 \node[below=5pt] at (U1) {$u_1(3)$};
\node[below=5pt] at (U2) {$u_4(1)$};
\node[below=5pt] at (U3) {$u_2(2)$};
\node[below=5pt] at (U4) {$u_5(2)$};
\node[below=5pt] at (U5) {$u_3(1)$};
\node[below=5pt] at (U6) {$u_6(3)$};
\end{tikzpicture}
\caption{The word $W'=\mathtt{111223122333}$ is a~shuffle square.}\label{fig5}
\end{figure}


\begin{example}\label{small1}\rm Let $$W=\underline{\textcolor{red}{\mathtt1}}\;\mathtt1\mathtt1\;\underline{\textcolor{red}{\mathtt0}}\underline{\textcolor{red}{\mathtt0}}\;\overline{\textcolor{blue}{\mathtt1}}
\overline{\textcolor{blue}{\mathtt0}}\overline{\textcolor{blue}{\mathtt0}}\;\mathtt0\;
\underline{\textcolor{red}{\mathtt1}}\;\overline{\textcolor{blue}{\mathtt1}}\;\mathtt0$$ and $(X,Y)$ be canonical twins in $W$ of length 4 given by $\supp(X)=\{1,4,5,10\}$
and $\supp(Y)=\{6,7,8,11\}$ (so, both equal $\mathtt{1001}$) as indicated. The corresponding matching $M$ and graph $G$ are given in Fig. \ref{fig6}. By adding a single loop at vertex $u_1$, we immediately obtain a~pair of twins of length 5 in $W$, $X'=Y'=\mathtt{11001}$, a bit different from that in Example~\ref{oper}.

\begin{figure}[h]
\centering
\begin{tikzpicture}[scale=1]
\coordinate (P1) at (0, 0);
\coordinate (P2) at (1, 0);
\coordinate (P3) at (2, 0);
\coordinate (P4) at (3, 0);
\coordinate (P5) at (4, 0);
\coordinate (P6) at (5, 0);
\coordinate (P7) at (6, 0);
\coordinate (P8) at (7, 0);
\coordinate (P9) at (8, 0);
\coordinate (P10) at (9, 0);
\coordinate (P11) at (10, 0);
\coordinate (P12) at (11, 0);

\node[above left=15pt] at (P1) {$M$:};

\draw[draw=green,thick,dashed,rounded corners] (-.2, -.2) rectangle (2.2, .2);
\draw[draw=green,thick,dashed,rounded corners] (2.8, -.2) rectangle (4.2, .2);
\draw[draw=green,thick,dashed,rounded corners] (4.8, -.2) rectangle (5.2, .2);
\draw[draw=green,thick,dashed,rounded corners] (5.8, -.2) rectangle (8.2, .2);
\draw[draw=green,thick,dashed,rounded corners] (8.8, -.2) rectangle (10.2, .2);
\draw[draw=green,thick,dashed,rounded corners] (10.8, -.2) rectangle (11.2, .2);

\foreach \P in {P1, P2, P3, P4, P5, P6, P7, P8, P9,P10,P11,P12}
 {\fill (\P) circle (2pt);}

\draw[-,bend left=60] (P1) to (P6);
\draw[-, bend left=90] (P4) to (P7);
\draw[-, bend left=90] (P5) to (P8);
\draw[-, bend left=90] (P10) to (P11);

\node[below=5pt] at (P1) {$1$};
\node[below=5pt] at (P2) {$2$};
\node[below=5pt] at (P3) {$3$};
\node[below=5pt] at (P4) {$4$};
\node[below=5pt] at (P5) {$5$};
\node[below=5pt] at (P6) {$6$};
\node[below=5pt] at (P7) {$7$};
\node[below=5pt] at (P8) {$8$};
\node[below=5pt] at (P9) {$9$};
\node[below=5pt] at (P10) {$10$};
\node[below=5pt] at (P11) {$11$};
\node[below=5pt] at (P12) {$12$};

\coordinate (U1) at (1, -2);
\coordinate (U2) at (5, -2);
\coordinate (U3) at (3.5, -3);
\coordinate (U4) at (7, -3);
\coordinate (U5) at (9.5, -2);
\coordinate (U6) at (11, -3);

\coordinate (UT) at (3.5, -3);
\node[left=120pt] at (UT) {$G$:};

\draw[-] (U1) to (U2);
\draw[-,red,thick,dashed] (U1) to[in=50,out=130, distance=1.5cm] (U1);
\draw[-, bend left=5] (U3) to (U4);
\draw[-, bend left=5] (U4) to (U3);
\draw[-] (U5) to[in=50,out=130, distance=1.5cm] (U5);

\foreach \P in {U1, U2, U3, U4, U5, U6}
 {\fill [fill=green, draw=black, line width=0.5pt] (\P) circle (2.5pt);}

 \node[below=5pt] at (U1) {$u_1(3)$};
\node[below=5pt] at (U2) {$u_3(1)$};
\node[below=5pt] at (U3) {$u_2(2)$};
\node[below=5pt] at (U4) {$u_4(3)$};
\node[below=5pt] at (U5) {$u_5(2)$};
\node[below=5pt] at (U6) {$u_6(1)$};
\end{tikzpicture}
\caption{Matching $M$ and graph $G$ for $W=\underline{\textcolor{red}{\mathtt1}}\;\mathtt1\mathtt1\;\underline{\textcolor{red}{\mathtt0}}\underline{\textcolor{red}{\mathtt0}}\;\overline{\textcolor{blue}{\mathtt1}}
\overline{\textcolor{blue}{\mathtt0}}\overline{\textcolor{blue}{\mathtt0}}\;\mathtt0\;
\underline{\textcolor{red}{\mathtt1}}\;\overline{\textcolor{blue}{\mathtt1}}\;\mathtt0$.}\label{fig6}
\end{figure}

As $W$ is even, we might wonder if it is a shuffle square, that is, if it possesses a pair of twins of length 6.
This can be quickly answered in negative, using the~graph representation as  a~tool. If $W$ were a shuffle square, there would exist an ordered graph $G$ satisfying conditions 1-3 of Proposition~\ref{Proposition Characterization} on vertex set $\{u_1,\dots,u_6\}$ with  $N_1=\{u_1,u_3,u_5\}$ and $N_2=\{u_2,u_4,u_6\}$. However, using Fig. \ref{fig6} again (but ignoring the edges) we see that such a~graph does not exist. Indeed, vertex $u_1$ cannot be connected to $u_5$ as then all edges at the bottom part would need to be adjacent to $u_6$ which is impossible. So, there must be a loop at $u_5$  which cuts off vertex $u_6$ at the bottom.
\end{example}
On a positive side, consider the following example.
\begin{example}\label{small2}\rm
Is  $$W=\mathtt{110}^5\mathtt{10}^{6}\mathtt{1110}^{4}\mathtt{110}^7$$ a shuffle square? One may try a tedious trial-and-error approach, but instead we draw a graph $G$ satisfying conditions 1-3 of Proposition~\ref{Proposition Characterization}, see Fig. \ref{fig7}. This confirms that $W$ is a shuffle square. Moreover, we may construct a pair of perfect twins in $W$ based on $G$. It is $$W=\underline{\textcolor{red}{\mathtt1}}\;\overline{\textcolor{blue}{\mathtt1}}\;\underline{\textcolor{red}{\mathtt0^5}}\underline{\textcolor{red}{\mathtt1}}
\underline{\textcolor{red}{\mathtt0}}\;\overline{\textcolor{blue}{\mathtt0^5}}\;\underline{\textcolor{red}{\mathtt1^2}}\;
\overline{\textcolor{blue}{\mathtt1}}\;\underline{\textcolor{red}{\mathtt0^3}}\;\overline{\textcolor{blue}{\mathtt01^2}}\;
\underline{\textcolor{red}{\mathtt0^2}}\;\overline{\textcolor{blue}{\mathtt0^5}}.$$

\begin{figure}[h]
\centering
\begin{tikzpicture}
\coordinate (U1) at (0,1.5);
\coordinate (U2) at (1,0);
\coordinate (U3) at (2,1.5);
\coordinate (U4) at (3,0);
\coordinate (U5) at (4,1.5);
\coordinate (U6) at (5,0);
\coordinate (U7) at (6,1.5);
\coordinate (U8) at (7,0);

\draw[-] (U1) to[in=50,out=130, distance=1.5cm] node[red,pos=0.85,right]{$1$} (U1);
\draw[-] (U3) to node[red,midway,above]{$1$} (U5);
\draw[-] (U5) to node[red,midway,above]{$2$} (U7);
\draw[-] (U2) to node[red,midway,above]{$5$} (U4);
\draw[-] (U4) to node[red,midway,above]{$1$} (U6);
\draw[-] (U6) to node[red,midway,above]{$3$} (U8);
\draw[-] (U8) to[in=50,out=130, distance=1.5cm] node[red,pos=0.85,right]{$2$} (U8);

\foreach \P in {U1, U2, U3, U4, U5, U6,U7,U8}
 {\fill [line width=0.5pt] (\P) circle (2.5pt);}

\node[below=5pt] at (U1) {$u_1(2)$};
\node[below=5pt] at (U2) {$u_2(5)$};
\node[below=5pt] at (U3) {$u_3(1)$};
\node[below=5pt] at (U4) {$u_4(6)$};
\node[below=5pt] at (U5) {$u_5(3)$};
\node[below=5pt] at (U6) {$u_6(4)$};
\node[below=5pt] at (U7) {$u_7(2)$};
\node[below=5pt] at (U8) {$u_8(7)$};
\end{tikzpicture}
\caption{The word
$W=\mathtt1\mathtt1\mathtt0^5\mathtt1\mathtt0\mathtt0^5\mathtt1^2\mathtt1\mathtt0^3\mathtt0\mathtt1^2\mathtt0^2\mathtt0^5$ is a~shuffle square.}\label{fig7}
\end{figure}

\end{example}

\section{Three easy pieces}\label{2pieces}

In this section we prove three relatively simple results on shuffle squares which demonstrate  the elegance and simplicity of the graphic representation.
An even word is \emph{dull} if all its runs have even length. It is called dull, because it is trivially a shuffle square with the perfect twins just taking a half of each run.

\subsection{Words with up to five runs}

Note that no even binary word with the total of \emph{three} runs is a shuffle square unless it is dull. Indeed, the only way to construct a nest-free graph $G$ with vertices $u_1<u_2<u_3$ of degrees $|U_1|,|U_2|,|U_3|$, respectively, and with no edges between $\{u_1,u_3\}$ and $u_2$, is to put $|U_i|/2$  loops at each vertex $u_i$, $i=1,2,3$, as otherwise nests would be created (see Fig. \ref{fig8}).

\begin{figure}[h]
\centering
\begin{tikzpicture}
\coordinate (U1) at (0,1.25);
\coordinate (U2) at (1.5,0);
\coordinate (U3) at (3,1.25);

\draw[-] (U1) to[in=50,out=130, distance=1.5cm] node[red,pos=0.85,right]{$\frac{a}2$} (U1);
\draw[-] (U2) to[in=50,out=130, distance=1.5cm] node[red,pos=0.85,right]{$\frac{b}2$} (U2);
\draw[-] (U3) to[in=50,out=130, distance=1.5cm] node[red,pos=0.85,right]{$\frac{c}2$} (U3);

\foreach \P in {U1, U2, U3}
 {\fill [line width=0.5pt] (\P) circle (2.5pt);}

\node[below=5pt] at (U1) {$u_1(a)$};
\node[below=5pt] at (U2) {$u_2(b)$};
\node[below=5pt] at (U3) {$u_3(c)$};
\end{tikzpicture}
\caption{The graph $G$ of a dull shuffle square with three runs.}\label{fig8}
\end{figure}.

We now settle the case of \emph{four} runs. The result has been already proved in~\cite{Gry2024}, though by a different method.

\begin{proposition}\label{abcd}
Let $a,b,c,d$ be positive integers, not all even, but with $a+c$ and $b+d$ both even. Then the~(even but not dull) word $W=\mathtt1^a\mathtt0^b\mathtt1^c\mathtt0^d$ is a shuffle square if and only if $a\geqslant c$ and $b\leqslant d$.
\end{proposition}

\begin{proof}
Consider a nest-free ordered graph on vertex set $u_1<u_2<u_3<u_4$ with capacities $a,b,c,d$, respectively, and edges are permitted only between $u_1$ and $u_3$, and between $u_2$~and~$u_4$; loops are allowed at any vertex.
If there is a shuffle square in $W$, then both edges, $\{u_1,u_3\}$ and $\{u_2,u_4\}$, must be present (with some multiplicities), as loops can only accommodate even degrees. But then there cannot be any loops either at $u_2$ or at $u_3$, implying that $a\geqslant c$ and $b\leqslant d$ (see Fig. \ref{fig9}).

\begin{figure}[h]
\centering
\begin{tikzpicture}
\coordinate (U1) at (0,1.5);
\coordinate (U2) at (2,0);
\coordinate (U3) at (4,1.5);
\coordinate (U4) at (6,0);

\draw[-] (U1) to[in=50,out=130, distance=1.5cm] node[red,pos=0.85,right]{$\frac{a-c}2$} (U1);
\draw[-] (U1) to node[red,midway,above]{$c$} (U3);
\draw[-] (U2) to node[red,midway,above]{$b$} (U4);
\draw[-] (U4) to[in=50,out=130, distance=1.5cm] node[red,pos=0.85,right]{$\frac{d-b}2$} (U4);

\foreach \P in {U1, U2, U3, U4}
 {\fill [line width=0.5pt] (\P) circle (2.5pt);}

\node[below=5pt] at (U1) {$u_1(a)$};
\node[below=5pt] at (U2) {$u_2(b)$};
\node[below=5pt] at (U3) {$u_3(c)$};
\node[below=5pt] at (U4) {$u_4(d)$};
\end{tikzpicture}
\caption{Graph $G$ for the shuffle square $\mathtt1^a\mathtt0^b\mathtt1^c\mathtt0^d$.}\label{fig9}
\end{figure}

On the other hand, if $a\geqslant c$ and $b\leqslant d$, then we draw edge $u_1u_3$ with multiplicity $c$, edge $u_2u_4$ with multiplicity $b$ and add $(a-c)/2$ loops at $u_1$ and $(d-b)/2$ loops at $u_4$ (see Fig. \ref{fig9}). The obtained ordered graph is nest-free and thus corresponds to a pair of perfect twins in $W$ of the~form
$$\underline{\textcolor{red}{\mathtt1^{(a+c)/2}}}\;\overline{\textcolor{blue}{\mathtt1^{(a-c)/2}}}\;\underline{\textcolor{red}{\mathtt0^b}}\;
\overline{\textcolor{blue}{\mathtt1^c}}
\;\underline{\textcolor{red}{\mathtt0^{(d-b)/2}}}\;\overline{\textcolor{blue}{\mathtt0^{(d+b)/2}}}.$$
Note that both twins form the word $\mathtt1^{(a+c)/2}\mathtt0^{(b+d)/2}$.
\end{proof}

\begin{remark}\label{cyclic}\rm
As a corollary of Proposition~\ref{abcd} we infer that every even word $\mathtt1^a\mathtt0^b\mathtt1^c\mathtt0^d$ with all $a,b,c,d$ positive, can be rotated (i.e., cyclically permuted), without breaking any of the four runs, so that the outcome is a shuffle square. Indeed, among the four possible rotations, $$abcd,\ bcda,\ cdab,\ dabc,$$
there is always one with the first element greater or equal to the third and the second element smaller or equal to the fourth.
So, at least one of them will satisfy the conditions in Proposition~\ref{abcd} (with $a,b,c,d$ renamed and $\mathtt0$ and $\mathtt1$ swapped, if needed). For instance,
$\mathtt1^7\mathtt0^6\mathtt1^9\mathtt0^8$ is not a shuffle square, but $\mathtt0^8\mathtt1^7\mathtt0^6\mathtt1^9$ is (as $8\geqslant 6$ and $7\leqslant 9$).
\end{remark}

A similar, though a bit more involved analysis can be carried out for words with \emph{five} runs, leading to the following characterization.

\begin{proposition}\label{abcde}
Let $a,b,c,d,e$ be positive integers, not all even, but with $a+c+e$ and $b+d$ both even. Then the word $W=\mathtt1^a\mathtt0^b\mathtt1^c\mathtt0^d\mathtt1^e$ is a shuffle square if and only if
\begin{itemize}
\item[(i)] $b=d$ and $c\leqslant a+e$, or
\item[(ii)] $b\leqslant d$ and $a\geqslant c$ and $e$ is even, or
\item[(iii)] $b\geqslant d$ and $c\leqslant e$ and $a$ is even.
\end{itemize}
\end{proposition}

\begin{proof}
 Suppose $W$ is a shuffle square. Then there is an ordered nest-free graph on vertices $u_1<u_2<u_3<u_4<u_5$ with edges only within $u_1,u_2,u_3$ and within $u_2,u_4$, and with degrees $a,b,c,d,e$, respectively. If there was a loop at $u_3$, then, consequently, there would be loops at all five vertices and no other edges, which contradicts the assumption that $W$ is dull.

Assume first that there are both, edge $\{u_1,u_3\}$ with multiplicity $c_1$ and $\{u_3,u_5\}$ with multiplicity $c_2$, $c=c_1+c_2$. Then there cannot be loops at $u_2$ or $u_4$ which implies that $b=d$. Moreover, we must have $a\geqslant c_1$ and $e\geqslant c_2$ and both differences $a-c_1$ and $e-c_2$ even. (Such a choice of $c_1$ and $c_2$ is possible owing to the assumption that $a+c+e$ is even.) Thus, we have arrived at case (i).

Next, assume that edge $\{u_1,u_3\}$ has multiplicity $c$ (and thus no edge $\{u_3,u_5\}$). This implies case (ii), as there is no loop at $u_2$ either. The last case is symmetric and leads to scenario (iii) (see Fig. \ref{fig10}).

On the other hand, if either of conditions (i), (ii), or (iii) holds, then one can easily create the desired graph  (see again Fig. \ref{fig10}).
\end{proof}

\begin{figure}[h]
\centering
\begin{tikzpicture}[scale=1]
\coordinate (U1) at (-2,8);
\coordinate (U2) at (-1,6.5);
\coordinate (U3) at (0,8);
\coordinate (U4) at (1,6.5);
\coordinate (U5) at (2,8);

\node at (-4,7.25) {\large{{\bf (i)}}};

\draw[-] (U1) to[in=50,out=130, distance=1.5cm] node[red,pos=0.15,left]{$\frac{a-c_1}2$} (U1);
\draw[-] (U1) to node[red,midway,above]{$c_1$} (U3);
\draw[-] (U2) to (U4);
\draw[-] (U3) to node[red,midway,above]{$c_2$} (U5);
\draw[-] (U5) to[in=50,out=130, distance=1.5cm] node[red,pos=0.85,right]{$\frac{e-c_2}2$} (U5);

\foreach \P in {U1, U2, U3, U4,U5}
 {\fill [line width=0.5pt] (\P) circle (2.5pt);}

\node[below=5pt] at (U1) {$u_1(a)$};
\node[below=5pt] at (U2) {$u_2(b)$};
\node[below=5pt] at (U3) {$u_3(c)$};
\node[below=5pt] at (U4) {$u_4(d)$};
\node[below=5pt] at (U5) {$u_5(e)$};

\node[red] at (0,5) {$c=c_1+c_2,\ \ c_1\leqslant a,\ \ c_2\leqslant e,\ \ a-c_1\equiv e-c_2\equiv0\,(\mbox{mod}\,2)$};

\coordinate (A1) at (-6,3.5);
\coordinate (A2) at (-5,2);
\coordinate (A3) at (-4,3.5);
\coordinate (A4) at (-3,2);
\coordinate (A5) at (-2,3.5);

\node at (-8,3) {\large{{\bf (ii)}}};

\draw[-] (A1) to[in=50,out=130, distance=1.5cm] node[red,pos=0.15,left]{$\frac{a-c}2$} (A1);
\draw[-] (A1) to node[red,midway,above]{$c$} (A3);
\draw[-] (A2) to node[red,midway,above]{$b$} (A4);
\draw[-] (A4) to[in=50,out=130, distance=1.5cm] node[red,pos=0.85,right]{$\frac{d-b}2$} (A4);
\draw[-] (A5) to[in=50,out=130, distance=1.5cm] node[red,pos=0.85,right]{$\frac{e}2$} (A5);

\coordinate (B1) at (2,3.5);
\coordinate (B2) at (3,2);
\coordinate (B3) at (4,3.5);
\coordinate (B4) at (5,2);
\coordinate (B5) at (6,3.5);

\node at (0,3) {\large{{\bf (iii)}}};

\draw[-] (B1) to[in=50,out=130, distance=1.5cm] node[red,pos=0.15,left]{$\frac{a}2$} (B1);
\draw[-] (B3) to node[red,midway,above]{$c$} (B5);
\draw[-] (B2) to node[red,midway,above]{$d$} (B4);
\draw[-] (B2) to[in=50,out=130, distance=1.5cm] node[red,pos=0.15,left]{$\frac{b-d}2$} (B2);
\draw[-] (B5) to[in=50,out=130, distance=1.5cm] node[red,pos=0.85,right]{$\frac{e-c}2$} (B5);

\foreach \P in {A1,A2,A3,A4,A5,B1,B2,B3,B4,B5}
 {\fill [line width=0.5pt] (\P) circle (2.5pt);}

\end{tikzpicture}
\caption{Variants of the graph $G$ of the shuffle square $\mathtt1^a\mathtt0^b\mathtt1^c\mathtt0^d\mathtt1^e$.}\label{fig10}
\end{figure}

\subsection{Words with separated ones}

In this section we consider even binary words with all $\mathtt1$'s separated from each other, that is, with all $\mathtt1$-runs of length 1, but with no restrictions on the $\mathtt0$'s (but see Remark~\ref{zero}).
That is, we are looking at even words of the form
\begin{equation}\label{word}
W=\mathtt{0}^{a_0}\mathtt{10}^{a_1}\mathtt{10}^{a_2}\cdots\mathtt{10}^{a_{2m}},
\end{equation}
for integers $m,a_1,\dots,a_{2m-1}\geqslant1$ and $a_0,a_{2m}\geqslant0$.
Our goal is to define conditions for the~values of  $a_i$ that make $W$ a shuffle square with the  $\mathtt1$'s \emph{belonging alternately to  perfect twins}.

For $m=1$, such conditions are given by Propositions~\ref{abcd} and~\ref{abcde} (with $\mathtt{0}$'s and $\mathtt{1}$'s swapped, if necessary). In particular, quite trivially, when $a_0=0$, the condition becomes $a_1\leqslant a_2$, when $a_2=0$ it becomes $a_0\geqslant a_1$, while if both $a_0=a_2=0$, $W$ is clearly not a shuffle square. But how about larger $m$?

This could be answered without referring to the ordered graphs, but we prefer the ``graphic'' approach as it seems to be more insightful.
Any graph representing perfect twins in word $W$, $m\geqslant2$, must attain a structure presented in Figure \ref{fig11}, with the additional constraint that for each $i=1,3,\dots,2m-3$, either $y_i=0$ or $z_{i+1}=0$. The~variables $x_i$, $i=0,\dots,2m-1$, $y_i$, $i=1,3,\dots, 2m-3$, and $z_i$, $i=0,2,,\dots, 2m$, represent multiplicities of the respective edges, so we have altogether $4m$ unknowns to find. The~top vertices have all capacities 1 and the bottom vertices, $v_0,v_1,\dots,v_{2m}$, have capacities $a_0,a_1,\dots,a_{2m}$. (In Fig.~\ref{fig11} the capacities are suppressed.)

\begin{figure}[h]
\centering
\begin{tikzpicture}[scale=.95]
\pgfresetboundingbox
  \path[use as bounding box] (-0.3,-2.2) rectangle (16.45,2.2);
\coordinate(V0) at (0,0);
\coordinate (U1) at (0.85,1.75);
\coordinate(V1) at (1.7,0);
\coordinate (U2) at (2.55,1.75);
\coordinate(V2) at (3.4,0);
\coordinate (U3) at (4.25,1.75);
\coordinate(V3) at (5.1,0);
\coordinate (U4) at (5.95,1.75);
\coordinate(V4) at (6.8,0);
\coordinate (U5) at (7.65,1.75);
\coordinate(V5) at (8.5,0);
\coordinate (U6) at (9.35,1.75);
\coordinate(V6) at (10.2,0);
\coordinate(VX) at (11.9,0);

\node at (11.475,0) {$\cdots$};
\node at (11.475,1.75) {$\cdots$};

\coordinate(VY) at (11.05,0);
\coordinate(V7) at (12.75,0);
\coordinate (U8) at (13.6,1.75);
\coordinate(V8) at (14.45,0);
\coordinate (U9) at (15.3,1.75);
\coordinate(V9) at (16.15,0);

\node[red] at (8.5,-1.75) {$ r_i=\deg^{(\rightarrow)}(v_i),\ \ \ \ \ell_i=\deg^{(\leftarrow)}(v_i),\ \ \ \ r_i+r_{i+1}=\ell_{i+1}+\ell_{i+2}$};

\draw[-] (U1) to node[red,midway,above]{$1$} (U2);
\draw[-] (U3) to node[red,midway,above]{$1$} (U4);
\draw[-] (U5) to node[red,midway,above]{$1$} (U6);
\draw[-] (U8) to node[red,midway,above]{$1$} (U9);

\draw[-] (V0) to[in=50,out=130, distance=1.5cm] node[red,midway,below]{$z_0$} (V0);
\draw[-] (V0) to node[red,midway,above=-3pt]{$x_0$} (V1);

\draw[-, bend left=90] (V1) to node[red,midway,above]{$y_1$} (V3);
\draw[-] (V1) to node[red,midway,above=-3pt]{$x_1$} (V2);
\draw[-] (V2) to[in=50,out=130, distance=1.5cm] node[red,midway,below]{$z_2$} (V2);
\draw[-] (V2) to node[red,midway,above=-3pt]{$x_2$} (V3);

\draw[-, bend left=90] (V3) to node[red,midway,above]{$y_3$} (V5);
\draw[-] (V3) to node[red,midway,above=-3pt]{$x_3$} (V4);
\draw[-] (V4) to[in=50,out=130, distance=1.5cm] node[red,midway,below]{$z_4$} (V4);
\draw[-] (V4) to node[red,midway,above=-3pt]{$x_4$} (V5);

\begin{scope}
\clip(8.5,0) rectangle (10.7,2);
\draw[-, bend left=90] (V5) to node[red,midway,above]{$y_5$} (VX);
\draw[-] (V6) to (VX);
\end{scope}
\draw[-] (V5) to node[red,midway,above=-3pt]{$x_5$} (V6);
\draw[-] (V6) to[in=50,out=130, distance=1.5cm] node[red,midway,below]{$z_6$} (V6);

\begin{scope}
\clip(12.25,0) rectangle (14.45,2);
\draw[-, bend left=90] (VY) to node[red,midway,above]{$y_{2m-3}$} (V8);
\draw[-] (VY) to (V7);
\end{scope}
\draw[-] (V7) to node[red,midway,above=-3pt]{$x_{2m-2}$} (V8);
\draw[-] (V7) to[in=50,out=130, distance=1.5cm] node[red,midway,below]{$z_{2m-2}$} (V7);
\draw[-] (V8) to node[red,midway,above=-3pt]{$x_{2m-1}$} (V9);
\draw[-] (V9) to[in=50,out=130, distance=1.5cm] node[red,midway,below]{$z_{2m}$} (V9);

\foreach \P in {U1,U2,U3,U4,U5,U6,U8,U9, V0,V1,V2,V3,V4,V5,V6,V7,V8,V9}
 {\fill [line width=0.5pt] (\P) circle (2.5pt);}

\node[below=5pt] at (V0) {$v_0$};
\node[below=5pt] at (V1) {$v_1$};
\node[below=5pt] at (V2) {$v_2$};
\node[below=5pt] at (V3) {$v_3$};
\node[below=5pt] at (V4) {$v_4$};
\node[below=5pt] at (V5) {$v_5$};
\node[below=5pt] at (V6) {$v_6$};
\node[below=5pt] at (V9) {$v_{2m}$};
\end{tikzpicture}
\caption{The structure of the graph representing perfect twins in the shuffle square $\mathtt{0}^{a_0}\mathtt{10}^{a_1}\mathtt{10}^{a_2}\cdots\mathtt{10}^{a_{2m}}$, $m\geqslant2$.}\label{fig11}
\end{figure}

One way to proceed from this would be to set up a system of $2m+1$ equations guaranteeing that for each $i=0,1,\dots,2m$, the degree of vertex $v_i$  equals $a_i$, supplemented by $m-1$ equations of the form $y_{2i-1}z_{2i}=0$, $i=1,2,\dots,m-1$. However, it seems simpler to apply more coarse setting in terms of the right degrees $r_i:=\deg_G^{(\rightarrow)}(v_i)$ and left degrees $\ell_i:=\deg_G^{(\leftarrow)}(v_i)$ and utilize the observation that $$r_0=\ell_0+\ell_1,\ \ \ \ \ r_{2m-1}+r_{2m}=\ell_{2m},\ \ \ \ \ \mbox{ and }\ \ \ \ \ r_{2i-1}+r_{2i}=\ell_{2i}+\ell_{2i+1},$$ for each $i=1,2,\dots, m-1$.

Thus, after getting rid of the $\ell_i$'s (as $r_i+\ell_i=a_i$ for all $i$), and solving for the $r_{2i}$'s,  we arrive at the following system of $m$ equations and just $2m-1$ variables.




\begin{equation}\label{syseq}\left\{\begin{array}{ll}
r_0=\frac12(a_0+a_1-r_1)&\\
r_{2i}=\frac12(a_{2i}+a_{2i+1}-r_{2i-1}-r_{2i+1}),&1\leqslant i\leqslant m-1\\
r_{2m}=\frac12(a_{2m}-r_{2m-1})
\end{array}\right.
\end{equation}
 As for all $i$ we have the natural constraints $0\leqslant r_i\leqslant a_i$, we thus have proved the following result. 

\begin{proposition}\label{only1}
An even binary word given by~\eqref{word}, with $m\geqslant2$, is a shuffle square with alternating $\mathtt1$'s if and only if the systems of equations~\eqref{syseq} admits  integer solutions with $0\leqslant r_i\leqslant a_i$ for all $i$. Then $W$ has perfect twins of the form $$\mathtt0^{r_0}\mathtt1\mathtt0^{r_1+r_2}\mathtt1\mathtt0^{r_3+r_4}\cdots \mathtt1\mathtt0^{r_{2m-1}+r_{2m}}=\mathtt0^{\ell_0+\ell_1}\mathtt1\mathtt0^{\ell_2+\ell_3}\mathtt1\mathtt0^{\ell_4+\ell_5}\cdots \mathtt1\mathtt0^{\ell_{2m}}.$$

\end{proposition}

\begin{remark}\label{zero}\rm
In fact, we may allow  $a_i=0$ also for $1\leqslant i\leqslant 2m-1$. Then the result becomes more general but less intuitive.

\begin{example}\label{a=0}\rm Consider the word
 $$W=\mathtt1\mathtt0^5\mathtt1\mathtt0^3\mathtt1\mathtt0^{4}\mathtt1\mathtt1\mathtt1\mathtt0^6\mathtt1\mathtt1\mathtt1\mathtt0^4\mathtt1\mathtt0^4.$$
 This is a word of type~\eqref{word} with
 $$a_1=5, a_2=3, a_3=4, a_4=a_5=0, a_6=6, a_7=a_8=0, a_9=4,a_{10}=4$$
 and the system~\eqref{syseq} has solutions
 $$r_1=5, r_2=1, r_3=r_4=r_5=0, r_6=3, r_7=r_8=0, r_9=4,r_{10}=0$$
 yielding perfect twins equal to $\mathtt1\mathtt0^6\mathtt1\mathtt1\mathtt0^3\mathtt1\mathtt1\mathtt0^{4}$ and with alternating $\mathtt1$'s:
  $$W=\underline{\textcolor{red}{\mathtt1\mathtt0^5}}\overline{\textcolor{blue}{\mathtt1}}\underline{\textcolor{red}{\mathtt0}}
  \overline{\textcolor{blue}{\mathtt0^2}}\underline{\textcolor{red}{\mathtt1}}\overline{\textcolor{blue}{\mathtt0^{4}}}
  \underline{\textcolor{red}{\mathtt1}}
 \overline{\textcolor{blue}{\mathtt1\mathtt1}}\underline{\textcolor{red}{\mathtt0^3}}\overline{\textcolor{blue}{\mathtt0^3}}
 \underline{\textcolor{red}{\mathtt1\mathtt1}}\overline{\textcolor{blue}{\mathtt1}}\underline{\textcolor{red}{\mathtt0^4}}
  \overline{\textcolor{blue}{\mathtt1\mathtt0^4}}.$$

\end{example}
\end{remark}

\medskip

For $m=2$, system~\eqref{syseq} becomes
$$\left\{\begin{array}{ll}\label{m=2}
r_0=\frac12(a_0+a_1-r_1)&\\
r_{2}=\frac12(a_{2}+a_{3}-a_1-r_{3})&\\
r_{4}=\frac12(a_{4}-r_{3}).&
\end{array}\right.$$
So, taking into account the additional constraints $0\leqslant r_i\leqslant a_i$, $i=0,1,2,3,4$,
the word $$W=\mathtt0^{a_0}\mathtt1\mathtt0^{a_1}\mathtt1\mathtt0^{a_2}\mathtt1\mathtt0^{a_3}\mathtt1\mathtt0^{a_4}$$ is a shuffle square (with the $\mathtt1$'s alternating) if there exist integers $r_1$ and $r_3$ such that
$$a_1-a_0\leqslant r_1\leqslant\,a_1+a_0, \;\; r_1\equiv\,a_1+a_0\quad (\!\!\!\!\!\!\mod 2)$$
and
\begin{equation}\label{maxmin}
\max(0,a_3-a_2-r_1)\leqslant r_3\leqslant\,\min(a_3+a_2-r_1,a_3,a_4),\;\;
r_3\equiv\,a_3+a_2-r_1 \quad(\!\!\!\!\!\!\mod 2).
\end{equation}

For $a_0=0$, we have $r_1=a_1$ and the constraints simplify even further, while the result becomes stronger, as we may drop the restriction to alternating $\mathtt1$'s.

\begin{corollary}\label{m=2,a_0=0}
An even binary word $$W=\mathtt1\mathtt0^{a_1}\mathtt1\mathtt0^{a_2}\mathtt1\mathtt0^{a_3}\mathtt1\mathtt0^{a_4}$$ with all $a_i\geqslant0$, $i=1,3,4$, and $a_2\geqslant1$ is a shuffle square  if and only if
\begin{equation}\label{2cond}
a_3+a_2-a_1\geqslant0\quad\mbox{ and }\quad a_4\geqslant a_3-a_2-a_1.
\end{equation}
\end{corollary}
\proof First observe that $W$ is a shuffle square with canonical twins of which the first twin takes the first two $\mathtt1$'s and the second twin - the last two $\mathtt1$'s \emph{if and only if} $a_1\leqslant a_3$ and $a_2+a_3=a_1+a_4$. Further, this pair of conditions implies ~\eqref{2cond}. Thus, $W$ is a shuffle square if and only if $W$ has perfect twins with alternating $\mathtt1$'s, which
by preceding discussion, holds if and only if
$$\max(0,a_3-a_2-a_1)\leqslant r_3\leqslant \min(a_3+a_2-a_1,a_3,a_4),$$
and $$r_3\equiv a_3+a_2-a_1 (\!\!\!\!\mod 2).$$ We are going to show that this is equivalent to~\eqref{2cond}. The necessary condition is straightforward. Assume now that ~\eqref{2cond} holds. If $a_4=a_3-a_2-a_1$, then $$\max(0,a_3-a_2-a_1)=\min(a_3+a_2-a_1,a_3,a_4)=a_4$$ and $$r_3=a_4\equiv a_3+a_2-a_1 (\!\!\!\!\mod 2).$$
Otherwise, owing to $a_2\geqslant1$, we have $$\max(0,a_3-a_2-a_1)<\min(a_3+a_2-a_1,a_3,a_4),$$ and so there are at least two possible values for $r_3$, allowing for parity adjustment. \qed


\begin{example}\label{4a}\rm
 Let $$W=\mathtt1\mathtt0^3\mathtt1\mathtt0^9\mathtt1\mathtt0^{11}\mathtt1\mathtt0^7.$$ We have $a_3+a_2-a_1=11+9-3=17$ and $a_3-a_2-a_1=-1$, so~\eqref{2cond} holds. Thus, system~\eqref{m=2} has  four solutions determined by the odd positive integers in $\{0,\dots,7\}$ (cf.~\eqref{maxmin}). They are, remembering that $r_1=a_1=3$,
 \begin{align*}&r_3=1, r_2=8, r_4=3
 \\&r_3=3, r_2=7, r_4=2,
 \\&r_3=5, r_2=6, r_4=1,
 \\&r_3=7, r_2=5, r_4=0
 \end{align*}
 and correspond, respectively, to the following four pairs of perfect twins in $W$:
 \begin{align*}
 (1)\;\;\;\;&\underline{\textcolor{red}{\mathtt1}}\underline{\textcolor{red}{\mathtt0^3}}\;\overline{\textcolor{blue}{\mathtt1}}\;\underline{\textcolor{red}{\mathtt0^8}}\;
 \overline{\textcolor{blue}{\mathtt0}}\;\underline{\textcolor{red}{\mathtt1}}\underline{\textcolor{red}{\mathtt0}}\;\overline{\textcolor{blue}{\mathtt0^{10}\mathtt1}}\;\underline{\textcolor{red}{\mathtt0^3}}\;\overline{\textcolor{blue}{\mathtt0^4}}&\text{with twins}\;\;\;\; &\mathtt1\;\mathtt0^{11}\;\mathtt1\;\mathtt0^4
  \\ (2)\;\;\;\;&
  \underline{\textcolor{red}{\mathtt1}}\underline{\textcolor{red}{\mathtt0^3}}\;\overline{\textcolor{blue}{\mathtt1}}\;\underline{\textcolor{red}{\mathtt0^7}}\;
 \overline{\textcolor{blue}{\mathtt0^2}}\;\underline{\textcolor{red}{\mathtt1}}\underline{\textcolor{red}{\mathtt0^3}}\;\overline{\textcolor{blue}{\mathtt0^{8}\mathtt1}}
\;\underline{\textcolor{red}{\mathtt0^2}}
 \;\overline{\textcolor{blue}{\mathtt0^5}}&\text{with twins}\;\;\;\;& \mathtt1\;\mathtt0^{10}\;\mathtt1\;\mathtt0^5
 \\ (3)\;\;\;\;&\underline{\textcolor{red}{\mathtt1}}\underline{\textcolor{red}{\mathtt0^3}}\;\overline{\textcolor{blue}{\mathtt1}}\;\underline{\textcolor{red}{\mathtt0^6}}
  \;\overline{\textcolor{blue}{\mathtt0^3}}\;\underline{\textcolor{red}{\mathtt1}}
 \underline{\textcolor{red}{\mathtt0^5}}\;\overline{\textcolor{blue}{\mathtt0^{6}\mathtt1}}\;\underline{\textcolor{red}{\mathtt0}}
 \;\overline{\textcolor{blue}{\mathtt0^6}}&\text{with twins}\;\;\;\;& \mathtt1\;\mathtt0^{9}\;\mathtt1\;\mathtt0^6
 \\ (4)\;\;\;\;
 &\underline{\textcolor{red}{\mathtt1}}\underline{\textcolor{red}{\mathtt0^3}}\;\overline{\textcolor{blue}{\mathtt1}}\;\underline{\textcolor{red}{\mathtt0^5}}\;
 \overline{\textcolor{blue}{\mathtt0^4}}\;\underline{\textcolor{red}{\mathtt1}}
 \underline{\textcolor{red}{\mathtt0^7}}\;\overline{\textcolor{blue}{\mathtt0^{4}\mathtt1\mathtt0^7}}&\text{with twins}\;\;\;\;& \mathtt1\;\mathtt0^{8}\;\mathtt1\;\mathtt0^7.
 \end{align*}
 The graphs corresponding to all four solutions are presented in Fig. \ref{fig12}. As $a_2+a_3\neq a_1+a_4$, there are no other perfect twins in $W$.
 \end{example}

 \begin{figure}[h]
\centering
\begin{tikzpicture}[scale=.95]

\node at (-8.4,2.5) {\large{{\bf (1)}}};

\coordinate (A1) at (-7.4,2);
\coordinate (A2) at (-6.6,.75);
\coordinate (A3) at (-5.8,2);
\coordinate (A4) at (-5,.75);
\coordinate (A5) at (-4.2,2);
\coordinate (A6) at (-3.4,.75);
\coordinate (A7) at (-2.6,2);
\coordinate (A8) at (-1.8,.75);

\foreach \P in {A1,A2,A3,A4,A5,A6,A7,A8}
 {\fill [line width=0.5pt] (\P) circle (2.5pt);}

\draw[-] (A1) to node[red,midway,above]{$1$} (A3);
\draw[-] (A5) to node[red,midway,above]{$1$} (A7);
\draw[-, bend left=65] (A2) to node[red,midway,above]{$2$} (A6);
\draw[-] (A2) to node[red,midway,below]{$1$} (A4);
\draw[-] (A4) to node[red,midway,below]{$8$} (A6);
\draw[-] (A6) to node[red,midway,below]{$1$} (A8);
\draw[-] (A8) to[in=50,out=130, distance=1.5cm] node[red,midway,above]{$3$} (A8);

\node[below=2pt] at (A2) {$(3)$};
\node[below=2pt] at (A4) {$(9)$};
\node[below=2pt] at (A6) {$(11)$};
\node[below=2pt] at (A8) {$(7)$};

\node at (0.8,2.5) {\large{{\bf (2)}}};

\coordinate (B1) at (1.8,2);
\coordinate (B2) at (2.6,.75);
\coordinate (B3) at (3.4,2);
\coordinate (B4) at (4.2,.75);
\coordinate (B5) at (5,2);
\coordinate (B6) at (5.8,.75);
\coordinate (B7) at (6.6,2);
\coordinate (B8) at (7.4,.75);

\foreach \P in {B1,B2,B3,B4,B5,B6,B7,B8}
 {\fill [line width=0.5pt] (\P) circle (2.5pt);}

\draw[-] (B1) to node[red,midway,above]{$1$} (B3);
\draw[-] (B5) to node[red,midway,above]{$1$} (B7);
\draw[-, bend left=65] (B2) to node[red,midway,above]{$1$} (B6);
\draw[-] (B2) to node[red,midway,below]{$2$} (B4);
\draw[-] (B4) to node[red,midway,below]{$7$} (B6);
\draw[-] (B6) to node[red,midway,below]{$3$} (B8);
\draw[-] (B8) to[in=50,out=130, distance=1.5cm] node[red,midway,above]{$2$} (B8);

\node[below=2pt] at (B2) {$(3)$};
\node[below=2pt] at (B4) {$(9)$};
\node[below=2pt] at (B6) {$(11)$};
\node[below=2pt] at (B8) {$(7)$};

\node at (-8.4,-.75) {\large{{\bf (3)}}};

\coordinate (C1) at (-7.4,-1.25);
\coordinate (C2) at (-6.6,-2.5);
\coordinate (C3) at (-5.8,-1.25);
\coordinate (C4) at (-5,-2.5);
\coordinate (C5) at (-4.2,-1.25);
\coordinate (C6) at (-3.4,-2.5);
\coordinate (C7) at (-2.6,-1.25);
\coordinate (C8) at (-1.8,-2.5);

\foreach \P in {C1,C2,C3,C4,C5,C6,C7,C8}
 {\fill [line width=0.5pt] (\P) circle (2.5pt);}

\draw[-] (C1) to node[red,midway,above]{$1$} (C3);
\draw[-] (C5) to node[red,midway,above]{$1$} (C7);
\draw[-] (C2) to node[red,midway,below]{$3$} (C4);
\draw[-] (C4) to node[red,midway,below]{$6$} (C6);
\draw[-] (C6) to node[red,midway,below]{$5$} (C8);
\draw[-] (C8) to[in=50,out=130, distance=1.5cm] node[red,midway,above]{$1$} (C8);

\node[below=2pt] at (C2) {$(3)$};
\node[below=2pt] at (C4) {$(9)$};
\node[below=2pt] at (C6) {$(11)$};
\node[below=2pt] at (C8) {$(7)$};

\node at (.8,-.75) {\large{{\bf (4)}}};

\coordinate (D1) at (1.8,-1.25);
\coordinate (D2) at (2.6,-2.5);
\coordinate (D3) at (3.4,-1.25);
\coordinate (D4) at (4.2,-2.5);
\coordinate (D5) at (5,-1.25);
\coordinate (D6) at (5.8,-2.5);
\coordinate (D7) at (6.6,-1.25);
\coordinate (D8) at (7.4,-2.5);

\foreach \P in {D1,D2,D3,D4,D5,D6,D7,D8}
 {\fill [line width=0.5pt] (\P) circle (2.5pt);}

\draw[-] (D1) to node[red,midway,above]{$1$} (D3);
\draw[-] (D5) to node[red,midway,above]{$1$} (D7);
\draw[-] (D2) to node[red,midway,below]{$3$} (D4);
\draw[-] (D4) to node[red,midway,below]{$4$} (D6);
\draw[-] (D6) to node[red,midway,below]{$7$} (D8);
\draw[-] (D4) to[in=50,out=130, distance=1.5cm] node[red,midway,above]{$1$} (D4);

\node[below=2pt] at (D2) {$(3)$};
\node[below=2pt] at (D4) {$(9)$};
\node[below=2pt] at (D6) {$(11)$};
\node[below=2pt] at (D8) {$(7)$};

\end{tikzpicture}
\caption{The possible graphs representing perfect twins in the shuffle square $\mathtt{1\;0}^{3}\mathtt{1\;0}^{9}\mathtt{1\;0}^{11}\mathtt{1\;0}^{7}$.}\label{fig12}
\end{figure}

\begin{example}\rm
Let $$W=\mathtt1\mathtt0^1\mathtt1\mathtt0^3\mathtt1\mathtt0^{2}\mathtt1\mathtt0^4.$$ This word satisfies both, condition~\eqref{2cond} and the system $a_1\leqslant a_3,\;a_2+a_3=a_1+a_4$. Consequently, it has both kinds of canonical twins: with alternating $\mathtt1$'s

$$\underline{\textcolor{red}{\mathtt1}}\underline{\textcolor{red}{\mathtt0}}\;\overline{\textcolor{blue}{\mathtt1}}\;
\underline{\textcolor{red}{\mathtt0}}\;
 \overline{\textcolor{blue}{\mathtt0^2}}\;\underline{\textcolor{red}{\mathtt1}}\underline{\textcolor{red}{\mathtt0^2}}\;
 \overline{\textcolor{blue}{\mathtt1}}\;\underline{\textcolor{red}{\mathtt0}}\;\overline{\textcolor{blue}{\mathtt0^3}},$$
and with consecutive $\mathtt1$'s
$$\underline{\textcolor{red}{\mathtt1}}\underline{\textcolor{red}{\mathtt0}}\;\underline{\textcolor{red}{\mathtt1}}\;
\underline{\textcolor{red}{\mathtt0^3}}\;
 \overline{\textcolor{blue}{\mathtt1}}\;\underline{\textcolor{red}{\mathtt0}}\overline{\textcolor{blue}{\mathtt0}}\;
 \overline{\textcolor{blue}{\mathtt1}}\;\overline{\textcolor{blue}{\mathtt0^4}}.$$
\end{example}

\begin{remark}\label{Inter}\rm Interestingly, if a word $$W=\mathtt{1\;0}^{a_1}\mathtt{1\;0}^{a_2}\mathtt{1\;0}^{a_{3}}\mathtt{1\;0}^{a_{4}}$$ with all $a_i\geqslant1$, $i=1,2,3,4$, violates condition~\eqref{2cond}, then there is always a~rotation so that the new word still begins with a $\mathtt1$, but now it is a~shuffle square.
 Indeed,  it suffices to redefine $a_1$ as the~smallest or second smallest of all four parameters, depending on which of them ``neighbors'' the largest parameter. Then $a_1\leqslant a_2+a_3$ and $a_2+a_4\geqslant a_3-a_1$, the latter because either $a_2$ or $a_4$ is the largest.

 More precisely, imagine a square with numbers $a_1,a_2,a_3,a_4$ occupying the four corners in the clockwise order (see Figure~\ref{square}), with the largest of them ($\ell$) sitting, say, in the left-top corner.  If the smallest number ($s$) sits at an adjacent corner, then let it be $a_1$ and so  $\ell=a_2$ or $\ell=a_4$. In either case both inequalities in~\eqref{2cond} hold. If $s$ sits opposite $\ell$, repeat the above, beginning with the second smallest one ($ss$) which now must be adjacent to $\ell$. Again, ~\eqref{2cond} holds (it is crucial that now $a_3$ is the second largest among the four numbers).
\end{remark}

\begin{figure}[h]
\centering
\begin{tikzpicture}
\def\x{2};\def\y{3.5};
\coordinate (a3L) at (0,0);
\coordinate (a2L) at (\x,0);
\coordinate (a1L) at (\x,\x);
\coordinate (a4L) at (0,\x);
\coordinate (a1R) at (\x+\y,0);
\coordinate (a4R) at (\x+\y+\x,0);
\coordinate (a3R) at (\x+\y+\x,\x);
\coordinate (a2R) at (\x+\y,\x);

\draw (a3L)--(a2L)--(a1L)--(a4L)--(a3L) (a1R)--(a4R)--(a3R)--(a2R)--(a1R);

\foreach \P in {a1L,a2L,a3L,a4L,a1R,a2R,a3R,a4R}
 {\fill [line width=0.5pt] (\P) circle (2.5pt);}

\node[above right] at (a3L) {\textcolor{red}{$a_3$}};
\node[above left] at (a2L) {\textcolor{red}{$a_2$}};
\node[below left] at (a1L) {\textcolor{red}{$a_1$}};
\node[below right] at (a4L) {\textcolor{red}{$a_4$}};

\node[above right] at (a1R) {\textcolor{red}{$a_1$}};
\node[above left] at (a4R) {\textcolor{red}{$a_4$}};
\node[below left] at (a3R) {\textcolor{red}{$a_3$}};
\node[below right] at (a2R) {\textcolor{red}{$a_2$}};

\node[above=5pt] at (a4L) {$\ell$};
\node[above=5pt] at (a1L) {$s$ or $ss$};
\node[above=5pt] at (a2R) {$\ell$};
\node[below=5pt] at (a1R) {$s$ or $ss$};
\end{tikzpicture}
\caption{Illustration to Remark \ref{Inter}.}\label{square}
\end{figure}

\begin{example}\rm
 For instance, let $$W=\mathtt{1\;0}^{16}\;\mathtt{1\;0}^{9}\;\mathtt{1\;0}^{4}\;\mathtt{1\;0}^{5}.$$ Here $$a_3+a_2-a_1=4+9-16<0,$$ so this word is not a shuffle square. But beginning at $\mathtt{1\;0}^{5}$, we get a new word $$W'=\mathtt{1\;0}^{5}\mathtt{1\;0}^{16}\mathtt{1\;0}^{9}\mathtt{1\;0}^{4},$$ a rotation of $W$, for which both inequalities in~\eqref{2cond} hold.
 Thus, $W'$ is a desired shuffle square.

\end{example}

\medskip

As a smooth transition to the next section, consider a class of binary words with all $\mathtt1$-runs of length one and all $\mathtt0$-runs of length at most two. Below we show that then, with one exception,  the conditions of Proposition~\ref{only1} are satisfied.
So, such words are shuffle squares, even with the additional property that the $\mathtt{1}$'s are alternating between the twins. (This can also be proved directly in terms of  nest-free graphs,  in a similar way as Proposition~\ref{Ths1} in the~next section.)

\begin{corollary}\label{1and2}
Every even binary word with all $\mathtt1$-runs of length one and all $\mathtt0$-runs of length at most two, except $\mathtt{1001}$, is a shuffle square.
\end{corollary}

\begin{proof} We present the proof only in the case when $$W=\mathtt{10}^{a_1}\mathtt{10}^{a_2}\cdots\mathtt{10}^{a_{2m}}$$ with all $a_i\in\{1,2\}$, $i=1,\dots,2m$, leaving the remaining cases to the reader.

For $m=1$, we are looking at $W=\mathtt{1\;0}^{a_1}\;\mathtt{1\;0}^{a_2}$ with $a_1$ and $a_2$ of the same parity, and thus with $a_1=a_2$. Then, trivially, $W$ is a shuffle square (in fact it is a square $W=(\mathtt{10^{a_1}})^2$).

 For $m\geqslant2$, we are going to check that the system of equations~\eqref{syseq}, with $a_0=r_0=0$ and $r_1=a_1$, has integer solutions $r_1,\dots,r_{2m}$ with $0\leqslant r_i\leqslant a_i$. This boils down to choosing the~values of variables $r_{2i-1}$, $i=1,\dots,m$, so that $r_1=a_1$,  the quantities $$a_{2i}+a_{2i+1}-r_{2i-1}-r_{2i+1},$$ $i=1,\dots, m-1$, as well as $a_{2m}-r_{2m-1}$ are all even, and $0\leqslant r_i\leqslant a_i$ for all $i$.

First observe that if $$a_{2i}+a_{2i+1}-r_{2i-1}-r_{2i+1}$$ is even for all $i=1,\dots, m-1$, then so is $a_{2m}-r_{2m-1}$. Indeed, recalling that $r_1=a_1$ and that $a_1$ and $-a_1$ have the same parity,
$$\sum_{i=1}^{m-1}(a_{2i}+a_{2i+1}-r_{2i-1}-r_{2i+1})\equiv \sum_{i=1}^{2m-1}a_i-r_{2m-1}\equiv 0\pmod 2.$$
Since $\sum_{i=1}^{2m}a_i\equiv 0\pmod 2$, we have $a_{2m}-r_{2m-1}\equiv 0 \pmod 2$, as desired.

Now, fix $1\leqslant i\leqslant m-1$,  assume that $r_1,\dots,r_{2i-1}$ have been chosen properly, and focus on the expression $$\rho_i:=a_{2i}+a_{2i+1}-r_{2i-1}-r_{2i+1}.$$
Since  $a_{2i},a_{2i+1}\in\{1,2\}$ and $0\leqslant r_{2i-1}\leqslant a_{2i-1}\leqslant2$, we have
$$0\leqslant a_{2i}+a_{2i+1}-r_{2i-1}\leqslant 4.$$
Hence, it is always possible to choose $r_{2i+1}\in\{0,1\}$ so that $\rho_i$ is even and nonnegative. Moreover, then $r_{2i+1}\leqslant1\leqslant a_{2i+1}$ and $r_{2i}\leqslant a_{2i}$. To see the latter, note that if $a_{2i}=1$, then $$a_{2i}+a_{2i+1}-r_{2i-1}\leqslant 3,$$ so $\rho_i\leqslant 2$ and $r_{2i}=\rho_i/2\leqslant1$, while if $a_{2i}=2$, then $r_{2i}=\rho_i/2\leqslant 2$.

Finally, for $i=m-1$, we need to check that also $$0\leqslant r_{2m}=\tfrac12(a_{2m}-r_{2m-1})\leqslant a_{2m}.$$ The~right-hand side inequality is obvious. As for the left-hand side, observe that by the above proof $r_{2m-1}\leqslant1$, so $r_{2m-1}\leqslant a_{2m}$.

\end{proof}

\section{The odd ABBA problem and beyond}\label{sec-abba}

Recall that if every run of a binary word $W$ has an even length, in particular, length two, then,  trivially, $W$~is a shuffle square. It is perhaps a bit surprising that a modest relaxation allowing runs of length one, but only for one of the two letters, say letter $\mathtt1$, becomes quite non-trivial.

One can quickly come up with such words which are not shuffle squares, like $\mathtt{1001}$ or $\mathtt{100110011001}=\mathtt{1001}^3$. Our main result, Theorem~\ref{abba}, proved in this section, implies that all words of the form $(\mathtt{1001})^n$, where $n$ is odd, are indeed \emph{not} shuffle squares.

But  first, as a sort of complementary result, we show that all even binary words with all $\mathtt0$-runs of length two and all $\mathtt1$-runs of length at most two, which are not of the form $(\mathtt{1001})^n$, $n$ odd, are shuffle squares. So, in fact, the ``odd ABBAs'' are the sole exception in this class.

\begin{proposition}\label{Ths1}
 Every even binary word in which all $\mathtt0$-runs have length two, while all $\mathtt1$-runs have length at most two is a shuffle square, unless it is of
the~form $(\mathtt{1001})^n$ for some odd~$n$.
\end{proposition}

\begin{proof} Let $W$ be an even binary word with every $\mathtt0$-run of length two and every $\mathtt1$-run of length at most 2. Assume that $W\neq(\mathtt{1001})^n$ for any odd $n$.
	
	We are going to construct a  corresponding nest-free graph $G$ on vertices $u_1,\dots, u_{m}$, where $m$ is the total number of runs in $W$. As always $G$ consists of two separated subgraphs, $G_1$ and $G_0$ whose vertices interlace and correspond to the $\mathtt1$-runs and $\mathtt0$-runs of $W$, respectively. Within $G_1$ we distinguish the vertices with capacity 1, call them \emph{singles}, and with capacity~2, call them \emph{doubles}.
   As  $G_1$ (the top graph in all figures) we take a union of paths connecting consecutive singles and going through all the doubles along the way, plus loops at all other doubles.

 Throughout the remainder of the proof, we will show that any such $G_1$ can be complemented by a suitable $2$-regular graph $G_0$ so that the union $G=G_0\cup G_1$ is nest-free.

We begin with two observations that can be verified by just looking at the corresponding illustrations.

   \begin{observation}\label{rescuer} If
$$P=(u_h,u_{h+2},\dots,u_{h+2s})$$ is a path of an odd length $s$ in $G_1$ and $h\geqslant2$, then there is a nest-free graph on $$\{u_{h-1},u_h,u_{h+1},\dots,u_{h+2s}\}$$ which, in addition to $P$, consists of a $2$-regular subgraph on $$\{u_{h-1},u_{h+1},\dots,u_{h+2s-1}\}.$$ An analogous statement is also
true in the symmetric case when $$h+2s<m$$ with $u_{h-1}$ replaced by $u_{h+2s+1}$.
   \end{observation}
For the proof see Fig. \ref{fig14}.

 \begin{figure}[h]
\centering
\begin{tikzpicture}[scale=1]
\coordinate (A1) at (0,0);
\coordinate (A2) at (0.75,1);
\coordinate (A3) at (1.5,0);
\coordinate (A4) at (2.25,1);
\coordinate (A5) at (3,0);
\coordinate (A6) at (3.75,1);
\coordinate (A7) at (4.5,0);
\coordinate (A8) at (5.25,1);
\coordinate (A9) at (6,0);
\coordinate(A10x) at (6.75,1);
\coordinate(A11x) at (7.5,0);

\node at (7.5,1.75) {$P$ ($s$ odd)};
\node at (7.5,1) {$\cdots$};
\node at (7.5,0) {$\cdots$};

\coordinate (A8x) at (7.5,0);
\coordinate (A9x) at (8.25,1);
\coordinate (A10) at (9,0);
\coordinate (A11) at (9.75,1);
\coordinate (A12) at (10.5,0);
\coordinate (A13) at (11.25,1);
\coordinate (A14) at (12,0);
\coordinate (A15) at (12.75,1);

\foreach \P in {A1,A2,A3,A4,A5,A6,A7,A8,A9,A10,A11,A12,A13,A14,A15}
 {\fill [line width=0.5pt] (\P) circle (2.5pt);}

\draw[-] (A2) to (A4);
\draw[-] (A4) to (A6);
\draw[-] (A6) to (A8);

\draw[-] (A11) to (A13);
\draw[-] (A13) to (A15);

\draw[-, bend left=30] (A1) to (A3);
\draw[-, bend right=30] (A1) to (A3);
\draw[-, bend left=30] (A5) to (A7);
\draw[-, bend right=30] (A5) to (A7);
\draw[-, bend left=30] (A12) to (A14);
\draw[-, bend right=30] (A12) to (A14);

\begin{scope}
\clip (5.25,-.5) rectangle (6.6,1);
\draw[-] (A8) to (A10x);
\draw[-, bend left=30] (A9) to (A11x);
\draw[-, bend right=30] (A9) to (A11x);
\end{scope}

\begin{scope}
\clip (8.4,-.5) rectangle (9.75,1);
\draw[-] (A9x) to (A11);
\draw[-, bend left=30] (A8x) to (A10);
\draw[-, bend right=30] (A8x) to (A10);
\end{scope}

\node[below=2mm] at (A1) {$u_{h-1}$};
\node[below=2mm] at (A2) {$u_h$};
\node[below=2mm] at (A3) {$u_{h+1}$};
\node[below=2mm] at (A4) {$u_{h+2}$};
\node[below=2mm] at (A5) {$u_{h+3}$};

\node[below=2mm] at (A14) {$u_{h+2s-1}$};
\node[below=2mm] at (A15) {$u_{h+2s}$};
\end{tikzpicture}
\caption{An illustration to the proof of Observation~\ref{rescuer}.}\label{fig14}
\end{figure}

   \begin{observation}\label{2-way-rescuer}
If $$P=(u_h,u_{h+2},\dots,u_{h+2s})\ \mbox{ and }\ Q=(u_{h+2s+2},\dots,u_{h+2s+2t})$$ are two consecutive odd paths $P$ and $Q$ in $G_1$, then there is a nest-free graph on $$\{u_h,u_{h+1},\dots,u_{h+2s+2t}\}$$ which, in addition to $P$ and $Q$, consists of a $2$-regular subgraph on $$\{u_{h+1},u_{h+3},\dots,u_{h+2s-1}, u_{h+2s+1},\dots,u_{h+2s+2t-1}\}.$$
   \end{observation}
For the proof see Fig. \ref{fig15}.

\begin{figure}[h]
\centering
\begin{tikzpicture}[scale=1]

\coordinate (A1) at (-6.5,1);
\coordinate (A2) at (-5.5,1);
\coordinate (A3) at (-4.5,1);
\coordinate(A3x) at (-3.5,1);

\node at (-3.5,1.4) {$P$ (odd)};
\node at (-3.5,1) {$\cdots$};

\coordinate (A4) at (-2.5,1);
\coordinate (A5) at (-1.5,1);
\coordinate (A6) at (-.5,1);
\coordinate (A7) at (.5,1);
\coordinate (A8) at (1.5,1);
\coordinate (A9) at (2.5,1);

\node at (3.5,1.4) {$Q$ (odd)};
\node at (3.5,1) {$\cdots$};

\coordinate(A10x) at (3.5,1);
\coordinate (A10) at (4.5,1);
\coordinate (A11) at (5.5,1);
\coordinate (A12) at (6.5,1);

\foreach \P in {A1,A2,A3,A4,A5,A6,A7,A8,A9,A10,A11,A12}
 {\fill [line width=0.5pt] (\P) circle (2.5pt);}

\draw[-] (A1) to (A2);
\draw[-] (A2) to (A3);
\begin{scope}
\clip (-4.5,.5) rectangle (-4.2,1.5);
\draw[-] (A3) to (A3x);
\end{scope}

\begin{scope}
\clip (-2.8,.5) rectangle (-2.5,1.5);
\draw[-] (A3x) to (A4);
\end{scope}
\draw[-] (A4) to (A5);
\draw[-] (A5) to (A6);

\draw[-] (A7) to (A8);
\draw[-] (A8) to (A9);
\begin{scope}
\clip (2.5,.5) rectangle (2.8,1.5);
\draw[-] (A9) to (A10x);
\end{scope}

\begin{scope}
\clip (4.2,.5) rectangle (4.5,1.5);
\draw[-] (A10x) to (A10);
\end{scope}
\draw[-] (A10) to (A11);
\draw[-] (A11) to (A12);

\coordinate (B1) at (-6,0);
\coordinate (B2) at (-5,0);

\node at (-4,0) {$\cdots$};

\coordinate (B3) at (-3,0);
\coordinate (B4) at (-2,0);
\coordinate (B5) at (-1,0);
\coordinate (B6) at (0,0);
\coordinate (B7) at (1,0);
\coordinate (B8) at (2,0);
\coordinate (B9) at (3,0);

\node at (4,0) {$\cdots$};

\coordinate (B10) at (5,0);
\coordinate (B11) at (6,0);

\foreach \P in {B1,B2,B3,B4,B5,B6,B7,B8,B9,B10,B11}
 {\fill [line width=0.5pt] (\P) circle (2.5pt);}

\draw[-, bend left=30] (B1) to (B2);
\draw[-, bend right=30] (B1) to (B2);
\draw[-, bend left=30] (B3) to (B4);
\draw[-, bend right=30] (B3) to (B4);

\draw[-] (B5) to (B6);
\draw[-] (B6) to (B7);
\draw[-, bend left=45] (B5) to (B7);

\draw[-, bend left=30] (B8) to (B9);
\draw[-, bend right=30] (B8) to (B9);
\draw[-, bend left=30] (B10) to (B11);
\draw[-, bend right=30] (B10) to (B11);

\end{tikzpicture}
\caption{An illustration to the proof of Observation~\ref{2-way-rescuer}.}\label{fig15}
\end{figure}

With the two observations in hand, we may quickly prove Proposition~\ref{Ths1} by induction on $N$, the number of paths in $G_1$. The~case $N=0$, that is when $W$ is a dull word, is trivial. For $N=1$, if the~sole path $P$ in $G_1$ has even length, then we draw double edges between consecutive pairs  of vertices of $G_0$ lying ``under $P$'' and loops elsewhere (see Fig. \ref{fig16}). Otherwise, as $W$ is not of the~form $(\mathtt{1001})^n$, $n$ odd, the conclusion follows by Observation~\ref{rescuer}.

\begin{figure}[h]
\centering
\begin{tikzpicture}
\coordinate (A1) at (-3.5,1);
\coordinate (A2) at (-2.5,1);
\coordinate (A3) at (-1.5,1);
\coordinate (A4) at (-.5,1);
\coordinate (A5) at (.5,1);
\coordinate (A6) at (1.5,1);
\coordinate (A7) at (2.5,1);
\coordinate (A8) at (3.5,1);

\foreach \P in {A1,A2,A3,A4,A5,A6,A7,A8}
 {\fill [line width=0.5pt] (\P) circle (2.5pt);}

\coordinate (B1) at (-3,0);
\coordinate (B2) at (-2,0);
\coordinate (B3) at (-1,0);
\coordinate (B4) at (0,0);
\coordinate (B5) at (1,0);
\coordinate (B6) at (2,0);
\coordinate (B7) at (3,0);

\foreach \P in {B1,B2,B3,B4,B5,B6,B7}
 {\fill [line width=0.5pt] (\P) circle (2.5pt);}

 \draw[-] (A1) to[in=50,out=130, distance=1.5cm] (A1);
\draw[-] (A2) to (A3);
\draw[-] (A3) to (A4);
\draw[-] (A4) to (A5);
\draw[-] (A5) to (A6);
\draw[-] (A7) to[in=50,out=130, distance=1.5cm] (A7);
\draw[-] (A8) to[in=50,out=130, distance=1.5cm] (A8);

\draw[-] (B1) to[in=50,out=130, distance=1.5cm] (B1);
\draw[-, bend left=30] (B2) to (B3);
\draw[-, bend right=30] (B2) to (B3);
\draw[-, bend left=30] (B4) to (B5);
\draw[-, bend right=30] (B4) to (B5);
\draw[-] (B6) to[in=50,out=130, distance=1.5cm] (B6);
\draw[-] (B7) to[in=50,out=130, distance=1.5cm] (B7);

\end{tikzpicture}
\caption{An illustration to the proof of Proposition~\ref{Ths1}, case $N=1$.}\label{fig16}
\end{figure}

Assume now that $N\geqslant 2$ and the statement is true for fewer than $N$ paths in $G_1$. Let $P$ and $Q$ be the first two (from left) paths in $G_1$. Let $p$ be the leftmost vertex of $P$ and let $q$ be the rightmost vertex of $Q$. If there are double vertices between $P$ and $Q$, then we may apply Observation~\ref{rescuer} separately to each  path if needed, that is, if it has an odd length (see Fig. \ref{fig17}). Otherwise, we apply Observation~\ref{2-way-rescuer}. In either case, one can build a
nest-free subgraph $G[P,Q]$ on the vertex set $\{p,p+1,\dots,q-1,q\}$ with all vertices of degrees equal to their capacities. After removing all the vertices of $G[P,Q]$, if there is still a nonempty remaining part, we apply the inductive assumption.

\begin{figure}[h]
\centering
\begin{tikzpicture}[scale=1]

\coordinate (A1) at (-6.5,1);
\coordinate (A2) at (-5.5,1);
\coordinate (A3) at (-4.5,1);
\coordinate (A4) at (-3.5,1);
\coordinate (A5) at (-2.5,1);
\coordinate (A6) at (-1.5,1);
\coordinate (A7) at (-.5,1);
\coordinate (A8) at (.5,1);
\coordinate (A9) at (1.5,1);
\coordinate (A10) at (2.5,1);
\coordinate (A11) at (3.5,1);
\coordinate (A12) at (4.5,1);
\coordinate (A13) at (5.5,1);
\coordinate (A14) at (6.5,1);

\foreach \P in {A1,A2,A3,A4,A5,A6,A7,A8,A9,A10,A11,A12,A13,A14}
 {\fill [line width=0.5pt] (\P) circle (2.5pt);}

\draw[-] (A1) to (A2) to (A3) to (A4) to (A5) to (A6);
\draw[-] (A7) to[in=50,out=130, distance=1.5cm] (A7);
\draw[-] (A8) to (A9) to (A10) to (A11) to (A12) to (A13) to (A14);

\node at (-4,1.5) {$P$ (odd)};
\node at (3.5,1.5) {$Q$ (even)};

\coordinate (B1) at (-6,0);
\coordinate (B2) at (-5,0);
\coordinate (B3) at (-4,0);
\coordinate (B4) at (-3,0);
\coordinate (B5) at (-2,0);
\coordinate (B6) at (-1,0);
\coordinate (B7) at (0,0);
\coordinate (B8) at (1,0);
\coordinate (B9) at (2,0);
\coordinate (B10) at (3,0);
\coordinate (B11) at (4,0);
\coordinate (B12) at (5,0);
\coordinate (B13) at (6,0);

\foreach \P in {B1,B2,B3,B4,B5,B6,B7,B8,B9,B10,B11,B12,B13}
 {\fill [line width=0.5pt] (\P) circle (2.5pt);}

\draw[-, bend left=30] (B1) to (B2);
\draw[-, bend right=30] (B1) to (B2);
\draw[-, bend left=30] (B3) to (B4);
\draw[-, bend right=30] (B3) to (B4);
\draw[-, bend left=30] (B5) to (B6);
\draw[-, bend right=30] (B5) to (B6);
\draw[-] (B7) to[in=50,out=130, distance=1.5cm] (B7);
\draw[-, bend left=30] (B8) to (B9);
\draw[-, bend right=30] (B8) to (B9);
\draw[-, bend left=30] (B10) to (B11);
\draw[-, bend right=30] (B10) to (B11);
\draw[-, bend left=30] (B12) to (B13);
\draw[-, bend right=30] (B12) to (B13);

\end{tikzpicture}
\caption{An illustration to the proof of Proposition~\ref{Ths1}, case $m\geqslant2$.}\label{fig17}
\end{figure}

\end{proof}

The proof of our main result, Theorem~\ref{abba}, relies on the following lemma.
\begin{lemma}\label{Lemma Path-Cycle}
	Let $H=P\cup C$ be an ordered graph on the set of vertices $\{1,2,\dots,n\}$, consisting of two vertex disjoint subgraphs: a path $P$ with endpoints at $1$ and $n$, and a cycle $C$. If $H$ is nest-free, then $C$ is an~\emph{even} cycle.
\end{lemma}
\begin{proof}
Let a nest-free graph $H$ satisfy the assumptions of the lemma and let $E(P)$ denote the set of edges of the path $P$. Consider an order relation $\prec$ on $E(P)$ defined as follows: for each pair of distinct edges $e,f\in E(P)$,  $e\prec f$ if $\min e\leqslant \min f$ and $\max e\leqslant \max f$. Since $P$ is nest-free, this relation is a~linear order and one can enumerate the~edges of $P$ accordingly to that order as $$E(P)=\{e_1\prec e_2\prec \cdots \prec e_p\}.$$ Let $\pi(e_i)\equiv i \pmod 2$, $i=1,\dots,p:=|E(P)|$, be the parity of edge $e_i$.

Let $V(C)=\{x_1,\dots,x_q\}$ be the vertex set of the cycle $C$, where the vertices are numbered in the order they are traversed by $C$ (we fix one of the two possible directions). We say that a vertex $x$ of $C$ is \emph{hugged} by an edge $e$ of $P$ if $\min e<x<\max e$. For each $i=1,\dots,q $, let $X_i=\{e\in P:\text{$x_i$ is hugged by}\; e\}$ and set $r_i=|X_i|$.

Observe that for each $i=1,\dots,q$ the number $r_i$ is odd, because path $P$, on the way from 1 to $n$, possibly zigzagging, ``passes'' each vertex of $C$ an odd number of times. Moreover, by the definition of hugging, each set $X_i$ consists of a block of consecutive edges of $P$  under $\prec$. Finally, by the nest-freeness of $H$, for each $i=1,\dots,q-1$, the sets $X_i$ and $X_{i+1}$ are disjoint and positioned ``side-to-side'', that is, $X_{i+1}$ either directly follows or directly precedes  $X_i$.

Thus, we may define a 2-coloring of $V(C)$ by assigning to $x_i\in V(C)$, $i=1,\dots,q$, color $\pi(e)$, where $e$ is the first (or the last) edge in $X_i$. Clearly, adjacent vertices on $C$ receive different colors, which implies that $q$ must be even. The proof is illustrated in Fig.~\ref{lem}.
\end{proof}

\begin{figure}[h]
\centering
\begin{tikzpicture}[scale=.9]
\coordinate(P1) at (0,0);
\coordinate(P2) at (3,0);
\coordinate(P3) at (6,0);
\coordinate(P4) at (9,0);
\coordinate(P5) at (7,1.3);
\coordinate(P6) at (4,1.3);
\coordinate(P7) at (7.5,2.6);
\coordinate(P8) at (10.5,0);

\coordinate(Q1) at (1.5,-1.3);
\coordinate(Q2) at (5,-1.3);
\coordinate(Q3) at (8,-1.3);
\coordinate(Q4) at (4.5,-2.6);

\foreach \Q in {Q1,Q2,Q3,Q4}
  {\draw[densely dashed, thin] (\Q |- 0,-3.5) -- (\Q |- 0,3.5);}
 
\draw[-, bend left=60] (P1) to node[midway,above]{$e_1$} (P2);
\draw[-, bend left=60] (P2) to node[midway,above]{$e_2$} (P3);
\draw[-, bend left=60] (P3) to node[midway,above]{$e_3$} (P4);
\draw[-, bend left=60] (P5) to node[midway,above]{$e_4$} (P4);
\draw[-, bend left=60] (P6) to node[midway,above]{$e_5$} (P5);
\draw[-, bend left=60] (P6) to node[midway,above]{$e_6$} (P7);
\draw[-, bend left=60] (P7) to node[midway,above]{$e_7$} (P8);
 
\draw[-, bend right=60] (Q1) to (Q2);
\draw[-, bend right=60] (Q2) to (Q3);
\draw[-, bend left=60] (Q3) to (Q4);
\draw[-, bend left=60] (Q4) to (Q1);

\foreach \P in {P1, P2, P3, P4, P5, P6, P7, P8}
 {\fill [fill=white, draw=black, line width=0.5pt] (\P) circle (3pt);};
 
\foreach \P in {Q1, Q2, Q3, Q4}
 {\fill [fill=gray, draw=black, line width=0.5pt] (\P) circle (3pt);};
 
\node[right] at (Q1) {$x_1$};
\node[right] at (Q2) {$x_2$};
\node[right] at (Q3) {$x_3$};
\node[left] at (Q4) {$x_4$};

\draw (15,-1) circle (1.5);
\coordinate(O1) at (15,0.5);
\coordinate(O2) at (16.5,-1);
\coordinate(O3) at (15,-2.5);
\coordinate(O4) at (13.5,-1);

\node[below] at (O1) {$x_1$};
\node[left] at (O2) {$x_2$};
\node[above] at (O3) {$x_3$};
\node[right] at (O4) {$x_4$};

\node[draw, rectangle, thick, inner sep=2pt] at (12.4,-1) {$\textcolor{blue}{e_2}\ \textcolor{red}{e_3}\ \textcolor{blue}{e_4}$};
\node[draw, rectangle, thick, inner sep=2pt] at (17.6,-1) {$\textcolor{blue}{e_2}\ \textcolor{red}{e_3}\ \textcolor{blue}{e_4}$};
\node[draw, rectangle, thick, inner sep=2pt] at (15,1) {$\textcolor{red}{e_1}$};
\node[draw, rectangle, thick, inner sep=2pt] at (15,-3) {$\textcolor{red}{e_5}\ \textcolor{blue}{e_6}\ \textcolor{red}{e_7}$};

\node[draw, rectangle, thick, inner sep=2pt] at (14.5,2.5) {$\textcolor{blue}{e_2}\ \textcolor{red}{e_3}\ \textcolor{blue}{e_4}$};
\node[draw, rectangle, thick, inner sep=2pt] at (13,2.5) {$\textcolor{red}{e_1}$};
\node[draw, rectangle, thick, inner sep=2pt] at (16.5,2.5) {$\textcolor{red}{e_5}\ \textcolor{blue}{e_6}\ \textcolor{red}{e_7}$};

\node at (13,3) {$x_1$};
\node at (14.5,3) {$x_2\ x_4$};
\node at (16.5,3) {$x_3$};

\foreach \P in {O1,O3}
 {\fill [fill=red, draw=black, line width=0.5pt] (\P) circle (3pt);};
 \foreach \P in {O2,O4}
 {\fill [fill=blue, draw=black, line width=0.5pt] (\P) circle (3pt);};
\end{tikzpicture}
\caption{An illustration to the proof of Lemma~\ref{Lemma Path-Cycle} ($p=7$, $q=4$). We have $\pi(e_i)$=1 (\textcolor{red}{red}) for $i$ odd and $\pi(e_i)$=0 (\textcolor{blue}{blue}) for $i$ even.}  \label{lem}
\end{figure}

Using this lemma we may easily prove our main result.

\begin{proof}[Proof of Theorem \ref{abba}] 
Suppose on the contrary that $W$ is a shuffle square and let $G=G_0\cup G_1$ be an ordered nest-free graph representing $W$ as in Proposition~\ref{Proposition Characterization}, where $G_0$ and $G_1$ are the~subgraphs of $G$ corresponding, respectively, to $\mathtt{0}$-runs and $\mathtt{1}$-runs. Thus, the degree sequence of $G_1$ is $(a_1,\ldots,a_{n+1})$ and the degree sequence of $G_0$ is $(b_1,\ldots,b_n)$. Since the latter sequence  does not split into two subsequences with the same sum, the graph $G_0$ is not bipartite, and so it contains an odd cycle $C$. On the other hand, in $G_1$ the two unique vertices of odd degree must belong to the same connected component and, therefore, $G_1$ contains a path $P$ whose ends are the first and the last vertex. It follows by Lemma \ref{Lemma Path-Cycle} that $G\supseteq H=C\cup P$ is not nest-free, a contradiction.
\end{proof}

\begin{corollary}\label{abba-cor}
Let $r\geqslant1$ and $m\geqslant1$ be  integers, and let $(a_1,\dots,a_{2m})$ be positive integers such that $a_1,a_{2m}$ are odd, while $a_2,\dots,a_{2m-1}$ are even. Then the word  $$W=\mathtt{1}^{a_1}\mathtt{0}^{r}\mathtt{1}^{a_2}\mathtt{0}^{r}\mathtt{1}^{a_3}\cdots\mathtt{1}^{a_{2m-1}}\mathtt{0}^{r}\mathtt{1}^{2m}$$ is not a shuffle square. In particular, the word $(\mathtt{1}\mathtt{0}^r\mathtt1)^{2m-1}$ is not a~shuffle square.
\end{corollary}
\begin{proof}  The sequence $(a_1,\dots,a_{2m})$ satisfies condition (1) of Theorem~\ref{abba}. Moreover, sequence $$(\underset{2m-1}{\underbrace{r,r,\dots,r}})$$ has an odd number of identical terms and thus satisfies condition (2) therein.
Hence, the main statement follows from Theorem~\ref{abba}. Setting $a_1=a_{2m}=1$  and $a_2=\cdots =a_{2m-1}=2$, we obtain the special case.
\end{proof}
Note that for $r=2$ we obtain the special class of odd $ABBA$s.


\section{Deletion distance to shuffle squares}

Given a~word $W$ of length $n$, let $f(W)$ be the length of the longest twins in $W$ and set $$g(W)=n-2f(W).$$ Further, let $g(n)=\max g(W)$ over all $W$, with $|W|=n$. By a~sophisticated approach via a~regularity lemma for words, Axenovich, Person, and Puzynina \cite{Axenovich-Person-Puzynina} showed that $g(n)=o(n)$ and, by providing a construction, that $g(n)=\Omega(\log n)$. The~lower bound was improved by Bukh \cite{Bukh} to $g(n)=\Omega(n^{1/3})$ by refining the construction from \cite{Axenovich-Person-Puzynina} (see \cite{BR} for a~proof). In a~personal communication Bukh has expressed his belief that, in fact, $g(n)=\Omega(\sqrt n)$ and that the real challenge would be to improve the upper bound, where one would need to refine the regularity approach from \cite{Axenovich-Person-Puzynina}. Focusing on the lower bound, in \cite{BR} we shared that belief and implicitly stated Conjecture~\ref{odd} which, in our notation says that  $g(n)=\Omega(\sqrt n)$.

We first describe a general class of words whose two particular instances yielded lower bounds on $g(n)$ in \cite{Axenovich-Person-Puzynina} and \cite{Bukh}. For positive integers $n_1,\dots,n_r$, let $$W(n_1,\dots,n_r)=U_1\cdots U_r$$ be the binary word whose consecutive runs $U_i$  have lengths $|U_i|=n_i$, $i=1,\dots,r$. We adopt the convention that $U_1$ is a $\mathtt1$-run. The~main features of the words used in \cite{Axenovich-Person-Puzynina} and \cite{Bukh} were that the runs rapidly decreased in length and that their lengths were odd.

An easy generalization of these two special words led to an observation that if $$n_1>n_2>\cdots >n_r$$ are all odd, then \begin{equation}\label{g}
g(W(n_1,\dots,n_r))\geqslant\min\{r,\delta\},
 \end{equation}
 where $\delta$ is the smallest difference between two consecutive terms (see \cite[Lemma 3.1]{BR}). The~idea of the proof is that if monotone twins take equally from each run, then they ``lose'' at least one element per run; otherwise, the leading twin is ahead at some point and the follower has to ``jump'' over the second next run, a loss that can never  be made up, due to the run lengths getting smaller and smaller.

The particular words used in \cite{Axenovich-Person-Puzynina} and \cite{Bukh} were $$A_r:=W(3^{r-1},\,3^{r-2},\,\dots,\,3^2,\,3,\,1)$$ and $$B_r:=W(r^2+1,\,(r-1)r+1,\,\dots,\, 2r+1,\,r+1),$$ respectively. For both, the minimum in~\eqref{g} is $r$.  Simple
calculations  reveal that for $A_r$ we have $r=\Theta(\log|A_r|)$ and for $B_r$, $r=\Theta(|B_r|^{1/3})$, both leading to the desired bounds. (For $n$ not expressible as $|A_r|$ or $|B_r|$ for any $r$ an obvious extrapolation has been applied.)

In view of the constructions from \cite{Axenovich-Person-Puzynina} and \cite{Bukh}  a natural candidate to
facilitate Conjecture~\ref{odd} seemed to be, for an odd integer $m$ and $r\leqslant(m+1)/2$, a binary word $$O_{m,r}:=W(m,\,m-2,\,m-4,\,\dots,\,m-2r+2)$$ consisting of $r$ \emph{consecutive} odd numbers. Indeed,
$$n:=|O_{m,r}|=r(m-r+1)=\Theta(m^2)$$ for $r=\Theta(m)$ and thus, showing that $g(O_{m,r})=\Omega(m)$ would do the job. However, inequality~\eqref{g} yields only $g(O_{m,r})\geqslant 2$, so a new approach would be needed.

 Below we show that, quite surprisingly,
 independently of $m$ and $r$, $O_{m,r}$  is, in fact, very close  to a shuffle square even for large $m$.

\begin{proposition}\label{close}
For every $m\geqslant 2r-1$, $m$ odd, the word $O_{m,r}$ contains twins of total length at least $|O_{m,r}|-23$, that is, $g(O_{m,r})\leqslant 23$.
\end{proposition}

\begin{proof}
Our aim is to find canonical twins in $O_{m,r}$ of double length at least $|O_{m,r}|-23$. To this end, we are going to construct a nest-free graph on $r$ vertices whose \emph{deficit} defined as $$|O_{m,r}|-\sum_{u\in V(G)}\deg_G(v)$$ is at most 23. We may and do assume that $r\geqslant24$, as otherwise we could draw as many as possible loops at every vertex, getting the deficit of $r\leqslant23$. For sheer convenience, we also assume that $r$ is divisible by 8. If not, we add loops at up to the~first seven vertices on, say, the left
end of $O_{m,r}$, procuring  a deficit of at most 7, and build $G$ on the~remaining vertices with deficit at most 16. So, let $r=8k$ for some integer $k\geqslant3$.

As usual, there are two kinds of vertices, say, red and blue, alternating, and there are no red-blue edges whatsoever. Let us identify the~vertices with their capacities, so we have the red set $$V_1=\{m,\,m-4,\,m-8,\,\dots,\,m-2r+4\}$$ and the blue set $$V_2=\{m-2,\,m-6,\,\dots,\,m-2r+2\}.$$

We begin by describing the ``red'' subgraph $G[V_1]$. Let us split $V_1$ into $k$ consecutive 4-tuples $$U_i=\{i,\,i-4,\,i-8,\,i-12\},$$ $i=m,m-16,\dots,m-2r+16$. Each $U_i$ spans the edge sets $$E_i=\{\ \{i,\,i-4\},\; \{i,\,i-8\},\;
\{i-4,\,i-12\}\ \}$$ with edges bearing  multiplicities $\mu(i,i-4)=8$, $\mu(i,i-8)=i-8$, and $\mu(i-4,i-12)=i-12$ (see Fig. \ref{fig19} - the top part). Note that this way the deficit of the red subgraph $G[V_1]$ is zero.
Indeed, for each $i$,
\begin{align*}
\deg_G(i)=&\,8+(i-8)=i,\\
\deg_G(i-4)=&\,8+(i-12)=i-4,\\
\deg_G(i-8)=&\,i-8,
\intertext{ and }
\deg_G(i-12)=&\,i-12.
\end{align*}

\begin{figure}[h]
\centering
\begin{tikzpicture}[scale=.9]
\coordinate (P1) at (0, 2.5);
\coordinate (P2) at (1, 2.5);
\coordinate (P3) at (2, 2.5);
\coordinate (P4) at (3, 2.5);
\coordinate (P5) at (4, 2.5);
\coordinate (P6) at (5, 2.5);
\coordinate (P7) at (6, 2.5);
\coordinate (P8) at (7, 2.5);
\coordinate (P9) at (8, 2.5);
\coordinate (P10) at (9, 2.5);
\coordinate (P11) at (10, 2.5);
\coordinate (P12) at (11, 2.5);
\node at (12,2.5) {$\cdots$};
\coordinate (P13) at (13, 2.5);
\coordinate (P14) at (14, 2.5);
\coordinate (P15) at (15, 2.5);
\coordinate (P16) at (16, 2.5);

\draw[-, bend left=90] (P1) to node[pos=.4,above]{\begin{tiny}$m-8$\end{tiny}} (P3);
\draw[-, bend left=90] (P2) to node[pos=.6,above]{\begin{tiny}$m-12$\end{tiny}} (P4);
\draw[-] (P1) to node[midway,above]{\begin{tiny}$8$\end{tiny}} (P2);

\draw[-, bend left=90] (P5) to (P7);
\draw[-, bend left=90] (P6) to (P8);
\draw[-] (P5) to (P6);

\draw[-, bend left=90] (P9) to (P11);
\draw[-, bend left=90] (P10) to (P12);
\draw[-] (P9) to (P10);

\draw[-, bend left=90] (P13) to (P15);
\draw[-, bend left=90] (P14) to (P16);
\draw[-] (P13) to (P14);

\foreach \P in {P1,P2,P3,P4,P5,P6,P7,P8,P9,P10,P11,P12,P13,P14,P15,P16}
 {\fill [fill=red, draw=black, line width=0.5pt] (\P) circle (2.5pt);}

\node[below=1.5mm] at (P1) {\begin{tiny}$m$\end{tiny}};
\node[below=3mm] at (P2) {\begin{tiny}$m-4$\end{tiny}};
\node[below=1mm] at (P3) {\begin{tiny}$m-8$\end{tiny}};
\node[below=3mm] at (P4) {\begin{tiny}$m-12$\end{tiny}};

\coordinate (Q1) at (.5, 0);
\coordinate (Q2) at (1.5, 0);
\coordinate (Q3) at (2.5, 0);
\coordinate (Q4) at (3.5, 0);
\coordinate (Q5) at (4.5, 0);
\coordinate (Q6) at (5.5, 0);
\coordinate (Q7) at (6.5, 0);
\coordinate (Q8) at (7.5, 0);
\coordinate (Q9) at (8.5, 0);
\coordinate (Q10) at (9.5, 0);
\coordinate (Q11) at (10.5, 0);
\node at (11.5,0) {$\cdots$};
\coordinate (Q12) at (12.5, 0);
\coordinate (Q13) at (13.5, 0);
\coordinate (Q14) at (14.5, 0);
\coordinate (Q15) at (15.5, 0);
\coordinate (Q16) at (16.5, 0);

\coordinate (Q10a) at (11.5,0);
\coordinate (Q10b) at (12.5,0);
\coordinate(Q12a) at (10.5,0);

\draw[-,red,bend left=90] (Q1) to node[pos=.4,above]{\begin{tiny}$m-10$\end{tiny}} (Q3);
\draw[-,bend left=90] (Q2) to node[pos=.8,above]{\begin{tiny}$m-14$\end{tiny}} (Q4);
\draw[-,bend left=90] (Q2) to node[midway,above]{\begin{tiny}$8$\end{tiny}} (Q5);
\draw[-,bend left=90] (Q5) to node[pos=.4,above]{\begin{tiny}$m-26$\end{tiny}} (Q7);
\draw[-,bend left=90] (Q6) to (Q8);
\draw[-,bend left=90] (Q6) to (Q9);
\draw[-,bend left=90] (Q9) to (Q11);
\begin{scope}
\clip (9.5,-.5) rectangle (10.75,2);
\draw[-, bend left=90] (Q10) to (Q10a);
\draw[-, bend left=90] (Q10) to (Q10b);
\end{scope}
\begin{scope}
\clip (12.25,-1) rectangle (13.5,0);
\draw[-, bend right=90] (Q12a) to (Q12);
\draw[-, bend right=90] (Q12a) to (Q13);
\end{scope}
\draw[-,bend left=90] (Q13) to (Q15);
\draw[-,red,bend left=90] (Q14) to node[pos=.6,above]{\begin{tiny}$m-2r+12$\end{tiny}} (Q16);

\foreach \P in {Q1,Q2,Q3,Q4,Q5,Q6,Q7,Q8,Q9,Q10,Q11,Q12,Q13,Q14,Q15,Q16}
 {\fill [line width=0.5pt] (\P) circle (2.5pt);}

\node[below=4mm] at (Q1) {\begin{tiny}$m-2$\end{tiny}};
\node[below=1mm] at (Q2) {\begin{tiny}$m-6$\end{tiny}};
\node[below=4mm] at (Q3) {\begin{tiny}$m-10$\end{tiny}};
\node[below=1mm] at (Q4) {\begin{tiny}$m-14$\end{tiny}};
\node[below=4mm] at (Q5) {\begin{tiny}$m-18$\end{tiny}};
\node[below=1mm] at (Q6) {\begin{tiny}$m-22$\end{tiny}};
\node[below=4mm] at (Q7) {\begin{tiny}$m-26$\end{tiny}};
\node[below=4mm] at (Q16) {\begin{tiny}$m-2r+2$\end{tiny}};

\end{tikzpicture}
\caption{Construction from the proof of Proposition~\ref{close}.}\label{fig19}
\end{figure}

The ``blue'' subgraph is a bit more complicated, as we have to avoid traps set up by the red graph. Nevertheless, its structure is almost the same: it consists of disjoint paths of lengths three, except that now these paths are slightly more stretched out, and, in addition, two isolated edges, one at the beginning and one at
the~end. It is these two edges what raises the deficit by 16. Precisely, we have
  $$\mu(m-2,m-10)=m-10$$ and $$\mu(m-2r+10,m-2r+1)=m-2r+1,$$ while the remaining vertices are broken up into $k-1$ 4-tuples $$W_i=\{i, i-8,i-12, i-20\},$$ $i=m-6,m-22,\dots,m-2r+26$, with edges sets $$E_i=\{\{i,i-8\}, \{i,i-12\}, \{i-12,i-20\}\}$$ and the edges of multiplicities
\begin{align*}
\mu(i,i-8)=&\,i-8,\\
\mu(i,i-12)=&\,8\\
\intertext{and}
\mu(i-12,i-18)=&\,i-18
\end{align*}
(see Fig. \ref{fig19} -- the bottom part). The~subgraphs $G[W_i]$ contribute no deficit, so that the total deficit of $G$ is at most $7+8+8=23$.
\end{proof}

To illustrate the above construction we present in Fig. \ref{fig20} the graph $G$
for the word
$$O_{47,24}=\mathtt1^{47}\mathtt0^{45}\mathtt1^{43}\mathtt0^{41}\mathtt1^{39}\mathtt0^{37}\mathtt1^{35}\mathtt0^{33}\mathtt1^{31}\mathtt0^{29}\mathtt1^{27}\mathtt0^{25}\mathtt1^{23}\mathtt0^{21}\mathtt1^{19}\mathtt0^{17}\mathtt1^{15}\mathtt0^{13}\mathtt1^{11}\mathtt0^{9}\mathtt1^{7}\mathtt0^{5}\mathtt1^{3}\mathtt0^{1}.$$

\begin{figure}[h]
\centering
\begin{tikzpicture}[scale=.95]
\coordinate (P1) at (0, 2);
\coordinate (P2) at (1.2, 2);
\coordinate (P3) at (2.4, 2);
\coordinate (P4) at (3.6, 2);
\coordinate (P5) at (4.8, 2);
\coordinate (P6) at (6, 2);
\coordinate (P7) at (7.2, 2);
\coordinate (P8) at (8.4, 2);
\coordinate (P9) at (9.6, 2);
\coordinate (P10) at (10.8, 2);
\coordinate (P11) at (12, 2);
\coordinate (P12) at (13.2, 2);

\draw[-, bend left=90] (P1) to node[midway,above]{$39$} (P3);
\draw[-, bend left=90] (P2) to node[midway,above]{$35$} (P4);
\draw[-] (P1) to node[pos=.6,above]{$8$} (P2);

\draw[-, bend left=90] (P5) to node[midway,above]{$23$} (P7);
\draw[-, bend left=90] (P6) to node[midway,above]{$19$} (P8);
\draw[-] (P5) to node[pos=.6,above]{$8$} (P6);

\draw[-, bend left=90] (P9) to node[midway,above]{$7$} (P11);
\draw[-, bend left=90] (P10) to node[midway,above]{$3$} (P12);
\draw[-] (P9) to node[pos=.6,above]{$8$} (P10);

\foreach \P in {P1,P2,P3,P4,P5,P6,P7,P8,P9,P10,P11,P12}
 {\fill [fill=red, draw=black, line width=0.5pt] (\P) circle (2.5pt);}

\node[red,below=2mm] at (P1) {$47$};
\node[red,below=2mm] at (P2) {$43$};
\node[red,below=2mm] at (P3) {$39$};
\node[red,below=2mm] at (P4) {$35$};
\node[red,below=2mm] at (P5) {$31$};
\node[red,below=2mm] at (P6) {$27$};
\node[red,below=2mm] at (P7) {$23$};
\node[red,below=2mm] at (P8) {$19$};
\node[red,below=2mm] at (P9) {$15$};
\node[red,below=2mm] at (P10) {$11$};
\node[red,below=2mm] at (P11) {$7$};
\node[red,below=2mm] at (P12) {$3$};

\coordinate (Q1) at (.6, .5);
\coordinate (Q2) at (1.8, .5);
\coordinate (Q3) at (3, .5);
\coordinate (Q4) at (4.2, .5);
\coordinate (Q5) at (5.4, .5);
\coordinate (Q6) at (6.6, .5);
\coordinate (Q7) at (7.8, .5);
\coordinate (Q8) at (9, .5);
\coordinate (Q9) at (10.2, .5);
\coordinate (Q10) at (11.4, .5);
\coordinate (Q11) at (12.6, .5);
\coordinate (Q12) at (13.8, .5);

\coordinate(81) at (.6,-1.2);
\coordinate(82) at (11.4,-1.2);

\node at (81) {$(-8)$};
\node at (82) {$(-8)$};
\draw[->, shorten >=2mm, shorten <=3mm] (81) -- (Q1);
\draw[->, shorten >=2mm, shorten <=3mm] (82) -- (Q10);

\draw[-,red,bend right=90] (Q1) to node[pos=.4,below]{$37$} (Q3);
\draw[-,bend right=90] (Q2) to node[pos=.6,above]{$33$} (Q4);
\draw[-,bend right=90] (Q2) to node[midway,below]{$8$} (Q5);
\draw[-,bend right=90] (Q5) to node[pos=.4,below]{$21$} (Q7);
\draw[-,bend right=90] (Q6) to node[pos=.6,above]{$17$} (Q8);
\draw[-,bend right=90] (Q6) to node[midway,below]{$8$} (Q9);
\draw[-,bend right=90] (Q9) to node[pos=.4,below]{$5$} (Q11);
\draw[-,red,bend right=90] (Q10) to node[pos=.6,below]{$1$} (Q12);

\foreach \P in {Q1,Q2,Q3,Q4,Q5,Q6,Q7,Q8,Q9,Q10,Q11,Q12}
 {\fill [line width=0.5pt] (\P) circle (2.5pt);}

\node[above=2mm] at (Q1) {$45$};
\node[above=2mm] at (Q2) {$41$};
\node[above=2mm] at (Q3) {$37$};
\node[above=2mm] at (Q4) {$33$};
\node[above=2mm] at (Q5) {$29$};
\node[above=2mm] at (Q6) {$25$};
\node[above=2mm] at (Q7) {$21$};
\node[above=2mm] at (Q8) {$17$};
\node[above=2mm] at (Q9) {$13$};
\node[above=2mm] at (Q10) {$9$};
\node[above=2mm] at (Q11) {$5$};
\node[above=2mm] at (Q12) {$1$};

\end{tikzpicture}
\caption{The graphic representation of the word $O_{47,24}$.}\label{fig20}
\end{figure}
The corresponding pair of canonical twins  of double length $24^2-16$
looks like this:
\begin{small}$$\underline{\textcolor{red}{\mathtt1^{47}}}\underline{\textcolor{red}{\mathtt0^{37}}}\mathtt0^8\underline{\textcolor{red}{\mathtt1^{35}}}\overline{\textcolor{blue}{\mathtt1^8}}
\underline{\textcolor{red}{\mathtt0^{41}}}
\overline{\textcolor{blue}{\mathtt1^{39}}}\overline{\textcolor{blue}{\mathtt0^{37}}}\overline{\textcolor{blue}{\mathtt1^{35}}}\overline{\textcolor{blue}{\mathtt0^{33}}}
\underline{\textcolor{red}{\mathtt1^{31}}}\underline{\textcolor{red}{\mathtt0^{21}}}\overline{\textcolor{blue}{\mathtt0^8}}\underline{\textcolor{red}{\mathtt1^{19}}}
\mathtt1^8\underline{\textcolor{red}{\mathtt0^{25}}}\overline{\textcolor{blue}{\mathtt1^{23}}}\overline{\textcolor{blue}{\mathtt0^{21}}}\overline{\textcolor{blue}{\mathtt1^{19}}}
\overline{\textcolor{blue}{\mathtt0^{17}}}
\underline{\textcolor{red}{\mathtt1^{15}}}\underline{\textcolor{red}{\mathtt0^{5}}}\mathtt0^8\underline{\textcolor{red}{\mathtt1^{3}}}
\overline{\textcolor{blue}{\mathtt1^8}}\underline{\textcolor{red}{\mathtt0^1}}\mathtt0^{8}\overline{\textcolor{blue}{\mathtt1^{7}}}\overline{\textcolor{blue}{\mathtt0^{5}}}
\overline{\textcolor{blue}{\mathtt1^{3}}}\overline{\textcolor{blue}{\mathtt0^{1}}},$$\end{small}
  with each single twin equal to
$$\mathtt1^{47}\mathtt0^{37}\mathtt1^{35}\mathtt0^{41}\mathtt1^{31}\mathtt0^{21}\mathtt1^{19}\mathtt0^{25}\mathtt1^{15}\mathtt0^5\mathtt1^3\mathtt0^1.$$
As you may see, the only gaps, each of size 8, appear in the second
run and the fifth one from the end. It would not be so easy to get these
twins without using graphs.

Finally, we complement inequality~\eqref{g} with a weaker statement which, however, applies to a broader family of words. It provides another sufficient condition for a word not to be a shuffle square, and so it shares the spirit of our main result, Theorem~\ref{abba}.
 In particular, it implies that the word $O_{m,r}$ defined above is not a shuffle square.

\begin{claim}\label{cl} Let $W=W(n_1,n_2,\dots,n_r)$. If $n_1$ is odd and
$n_2>\cdots> n_r$  form a decreasing sequence (but are not necessarily
odd), then $W$ is not a shuffle square.
\end{claim}
\begin{proof}
We will show that a graph corresponding to canonical perfect twins does not exists. Such a graph if exists has vertices $u_1,\dots,u_r$ with degrees $\deg_G(u_i)=n_i$ and is nest-free. Since $n_1$ is odd, there must be an edge from $u_1$ to the right (i.e., not just loops at $u_1$). This implies that there are no loops at $u_2$ and, thus, $\deg_G^{(\leftarrow)}(u_2)=0$ (see Fig. \ref{fig18} - the left side). Let $u_h$ be the~rightmost vertex with $\deg_G^{(\leftarrow)}(u_h)=0$. To have $\deg_G(u_h)=n_h$, vertex $u_h$ must send edges to at least two vertices to the right, say, $u_i$ and $u_j$, $h<i<j$, all indices with the same parity. But then, owing to the nestlessness of $G$, we have $\deg_G^{(\leftarrow)}(u_{i+1})=0$, a~contradiction with the definition of $u_h$ (see Fig. \ref{fig18} - the right side).

\begin{figure}[h]
\centering
\begin{tikzpicture}

\coordinate (U1) at (0,2);
\coordinate (U2) at (1,0);
\coordinate (U3) at (2,2);
\coordinate (U4) at (3,0);
\coordinate (Uh) at (6,0);

\coordinate(U2+) at (5,0);
\coordinate (U3+) at (4,2);
\coordinate (Uh-) at (2,0);

\node at (4.5,0) {$\cdots$};
\node at (7,0) {$\cdots$};
\node at (9,0) {$\cdots$};
\node at (11,0) {$\cdots$};

\node at (4.5,2) {$\cdots$};
\node at (11,2) {$\cdots$};

\coordinate (Ui) at (8,0);
\coordinate (Ui+1) at (9,2);
\coordinate (Uj) at (10,0);

\coordinate (Ui+1-) at (5,2);
\coordinate (Ui+1+) at (13,2);

\foreach \P in {U1,U2,U3,U4,Uh,Ui,Ui+1,Uj}
 {\fill [line width=0.5pt] (\P) circle (2.5pt);}

\draw[-] (U1) to[in=50,out=130, distance=1.5cm] (U1);
\draw[-] (U2) to[in=50,out=130, distance=1.5cm] (U2);
\draw[-] (Uh) to[in=50,out=130, distance=1.5cm] (Uh);
\draw[-] (Ui+1) to[in=50,out=130, distance=1.5cm] (Ui+1);

\draw[-, bend left=30] (Uh) to (Ui);
\draw[-, bend left=30] (Uh) to (Uj);

\begin{scope}
\clip (1,-.5) rectangle (2,.5);
\draw[-, bend left=30] (U2) to (U2+);
\end{scope}
\begin{scope}
\clip (5,-.5) rectangle (6,.5);
\draw[-, bend left=30] (Uh-) to (Uh);
\end{scope}
\begin{scope}
\clip (0,1.5) rectangle (1,2.5);
\draw[-, bend left=30] (U1) to (U3+);
\end{scope}
\begin{scope}
\clip (8,1.5) rectangle (10,2.5);
\draw[-, bend left=30] (Ui+1-) to (Ui+1) to (Ui+1+);
\end{scope}

\node[below=2mm] at (U1) {$u_1$};
\node[below=2mm] at (U2) {$u_2$};
\node[below=2mm] at (U3) {$u_3$};
\node[below=2mm] at (U4) {$u_4$};
\node[below=2mm] at (Uh) {$u_h$};
\node[below=2mm] at (Ui) {$u_i$};
\node[below=2mm] at (Uj) {$u_j$};
\node[below=2mm] at (Ui+1) {$u_{i+1}$};

\draw[red, line width=.5pt] (8.6,2.3) -- (9.4,2.7);
\draw[red, line width=.5pt] (8.6,2.7) -- (9.4,2.3);
\draw[red, line width=.5pt] (.6,.3) -- (1.4,.7);
\draw[red, line width=.5pt] (.6,.7) -- (1.4,.3);
\draw[red, line width=.5pt] (5.6,.3) -- (6.4,.7);
\draw[red, line width=.5pt] (5.6,.7) -- (6.4,.3);

\draw[red, line width=.5pt] (5.1,0) -- (5.5,.6);
\draw[red, line width=.5pt] (5.1,.6) -- (5.5,0);
\draw[red, line width=.5pt] (8.1,2) -- (8.5,2.6);
\draw[red, line width=.5pt] (8.1,2.6) -- (8.5,2);

\end{tikzpicture}
\caption{An illustration to the proof of Lemma~\ref{cl}.}\label{fig18}
\end{figure}
\end{proof}

\begin{example}\rm
Consider even words $\mathtt{10}^9\mathtt1^7\mathtt0^5\mathtt1^4\mathtt0^3\mathtt1^2\mathtt0$ and   $\mathtt1^{11}\mathtt0^8\mathtt1^6\mathtt0^4\mathtt1^3$. As both satisfy the~assumptions of Claim~\ref{cl}, none of them is a shuffle square.
\end{example}

\section{Open Problems}

We conclude the paper with some directions and open problem for future consideration.

\subsection{Doubly binary words}
Call a binary word $W=U_1\cdots U_m$ \emph{doubly binary} if the lengths of runs, $|U_1|,\dots, |U_m|$, also form a binary sequence (over alphabet $\{1,2\}$). There exist many examples of such words, one notable instance being the famous \emph{Thue-Morse word},
$$T=\mathtt{01101001100101101001011001101001}...,$$
which is an infinite concatenation of two blocks, $A=\mathtt{0110}$ and $B=\mathtt{1001}$, in the order designated by the sequence $T$ itself, that is, by the sequence $$ABBABAABBAABABBA...\ .$$ Note that the only prefixes of $T$ which are even words are those of length $4k$ for $k\ge1$.
It turns out that for $k\geqslant 3$ they are actually shuffle squares. We leave an easy proof to the interested reader, as a kind of exercise in applying the graphic method introduced in this paper. (As a hint: use induction on $k$ and consider two subcases  w/r to whether the $k$-th block is identical to the $(k-1)$-st block, or not.)

Another self-defining doubly binary sequence is the mysterious Kolakoski word (\cite{Kola}, see also  "Sequence A000002'' in OEIS) which is equal to the sequence of the lengths of its runs:

$$K=\mathtt{122112122122112112212112122112112...}$$
 The prefixes of $K$ which are even words are rather rare and finding shuffle squares among them may be more challenging.
 The first two even prefixes of $K$ are
 $$\mathtt{12211212}\quad\mbox{and}\quad\mathtt{1221121221221121},$$
  and neither is a shuffle square (the largest twins in either prefix leave out each just two gaps).

 \begin{question} Is there at least one prefix of the Kolakoski sequence which is a shuffle square?
\end{question}

 The most ambitious goal in this direction is the following.

\begin{problem} Characterize all doubly binary even words which are \emph{not} shuffle squares.
\end{problem}

\subsection{Deletion distance to shuffle squares}

It would be nice to know more about the~maximum deletion distance of a binary word of length $n$ to a shuffle square, which we denoted by $g(n)$. One natural try to get close to the conjectured value of $g(n)=\Omega(\sqrt{n})$ failed, as demonstrated in Proposition~\ref{close}, but this does not mean that Conjecture~\ref{odd} is false. In search for better constructions, one should perhaps allow varying differences between consecutive, still odd, terms.

A closely related is the following intriguing conjecture, mentioned in the introduction, posed by He, Huang, Nam, and Thaper in \cite{He2021}.

\begin{conjecture}
Almost all even binary words are shuffle squares.
\end{conjecture}

This means that the proportion of shuffle squares among all even binary words of length $n$ tends to one with $n$ tending to infinity.  This conjecture implies that a~random binary word can be made a shuffle square by deleting at most \emph{two} letters. Indeed, for $n$ even, construct a bijection between \emph{even} and \emph{non-even} words by flipping the last letter; for $n$ odd, construct a bijection between
\emph{all} words of length $n$ and \emph{even} words of length $n+1$ by adding an appropriate letter at the end. (This was already observed in \cite{DGR} for $n$ even, and in \cite{He2021} for all $n$, but with distance at most three rather than just two.)

\subsection{Cutting distance to shuffle squares}

Let $c(W)$ be the minimum integer $c$ such that there exists a partition $W=B_1\cdots B_{c+1}$ into consecutive blocks and a permutation $\pi$ of $[c+1]$ for which $B_{\pi(1)}\cdots B_{\pi(c+1)}$ is a shuffle square. In other words, $c(W)$ is the minimum number of cuts applied to $W$ after which it is possible to reassemble the resulting blocks to form a shuffle square. In particular, $c(W)=1$ means that a shuffle square can be obtained from $W$ by a cyclic permutation of its~letters.



For instance, as discussed before (see Remark~\ref{cyclic} after Proposition \ref{abcd}), for every even binary word $W$ with at most four runs we have $c(W)\le1$. On the other hand,  as checked by a computer, the word $$P=\mathtt{111110110000111100010000}=\mathtt{1}^5\mathtt{0}\mathtt{1}^2\mathtt{0}^4\mathtt{1}^4\mathtt{0}^3\mathtt{1}\mathtt{0}^4$$
is not a shuffle square, and the same is true for any of its cyclic permutations. However, it can be cut into three pieces, $$\mathtt{111\;|\;1101100001111000\;|\;10000},$$ which may be rearranged to give a shuffle square: $$\mathtt{1101100001111000\;|\;111\;|\;10000}=\mathtt{1}^2\mathtt{0}\mathtt{1}^2\mathtt{0}^4\mathtt{1}^4\mathtt{0}^3\mathtt{1}^4\mathtt{0}^4.$$
Thus,  $c(P)=2$ for the word $P$ above. Is it true that two cuts are always sufficient for even binary words?

\begin{conjecture}\label{cW2}
	Every even binary word $W$ satisfies $c(W)\leqslant 2$.
\end{conjecture}

Let us remark that at present we do not even have a proof that $c(W)\le C$ for some constant $C$.
Also, Conjecture~\ref{cW2}  does not hold for ternary words, as one may easily check that $$c(\mathtt{122231113332})=3.$$
 We dare, however, to formulate the following generalization.

\begin{conjecture}
	For each $k\geqslant2$, every even $k$-ary word $W$ satisfies $c(W)\leqslant k$.
\end{conjecture}

\subsection{Reverse shuffle squares}
D. Henshall, N. Rampersad, J. Shallit in \cite{Hen2012} defined also the notion of a \emph{reverse shuffle square}, that is, a word which can be split into two subwords which are \emph{reverses} of each other (cf. \cite{He2021}). For instance,
$$\underline{\textcolor{red}{\mathtt{10}}}\overline{\textcolor{blue}{\mathtt{011}}}\underline{\textcolor{red}{\mathtt{01}}}
\overline{\textcolor{blue}{\mathtt0}}
\underline{\textcolor{red}{\mathtt1}}\overline{\textcolor{blue}{\mathtt{01}}}
\underline{\textcolor{red}{\mathtt0}}$$
is a reverse shuffle square, consisting of two copies of $\mathtt{100110}$, one straight and one reverse.

We believe that a suitable modification of Proposition~\ref{Proposition Characterization} will provide a strong tool to study the existence of and the deletion distance from reverse shuffle squares. Specifically, for binary
words it seems that condition 3 (nestlessness) should be replaced by
forbidding alignments, that is, pairs of edges $e,f$ with $\max e<\min f$.

Finally, in the context of Conjecture~\ref{cW2}, one can  easily deduce from the famous Necklace Splitting Theorem (see \cite{Alon}, \cite{Alon-West}, \cite{Gry2023}) that every even binary word is at cutting distance at most two from a \emph{reverse shuffle square}.

\section*{Acknowledgment} We would like to thank Andrzej Komisarski for sharing with us his lovely ABBA problem, as well as some results of his impressive computer experiments.

\end{document}